\documentclass[11pt]{article}
\usepackage[utf8]{inputenc}
\usepackage[margin= 2.5 cm, bottom=20mm,footskip=10mm,top=20mm]{geometry}
\usepackage{setspace}

\usepackage{amsthm}
\usepackage{amsmath}
\usepackage{amssymb}
\usepackage{amsfonts}
\usepackage[noadjust]{cite}
\usepackage{epsfig}
\usepackage{mathtools}
\usepackage{enumitem}
\usepackage{stackengine,graphicx}
\usepackage{tikz}
\usetikzlibrary{calc}
\usetikzlibrary{shapes.geometric}
\usetikzlibrary{positioning,quotes}
\usetikzlibrary{decorations.pathmorphing}
\usepackage{comment}

\usepackage{microtype}
\usepackage{hyperref}

\newtheorem{thm}{Theorem}[section]
\newtheorem{lem}[thm]{Lemma}
\newtheorem{prop}[thm]{Proposition}

\newtheorem{conj}[thm]{Conjecture}
\newtheorem{qu}[thm]{Question}
\newtheorem{coro}[thm]{Corollary}

\newtheorem{claim}{Claim}
\newtheorem{subclaim}{Observation}
\newtheorem{fact}[thm]{Fact}

\newtheorem{defi}[thm]{Definition}
\theoremstyle{definition}

\theoremstyle{plain}

\newcounter{proofctr}

\pretocmd{\proof}{%
  \stepcounter{proofctr}%
  \setcounter{claim}{0}%
}{}{}

\newenvironment{claimproof}{
  \begingroup
  \let\origqed\qedsymbol
  \renewcommand{\qedsymbol}{$\blacktriangleleft$}
  \par\noindent\textit{Proof of Claim \theclaim.}
}{
  \qed
  \par\medskip
  \let\qedsymbol\origqed
  \endgroup
}

\newcounter{claimproofctr}

\pretocmd{\claimproof}{%
  \stepcounter{claimproofctr}%
  \setcounter{subclaim}{0}%
}{}{}

\newenvironment{subclaimproof}{
  \begingroup
  \let\origqed\qedsymbol
  \renewcommand{\qedsymbol}{$\blacklozenge$}
  \par\noindent\textit{Proof of Observation~\thesubclaim.}
}{
  \qed
  \par\medskip
  \let\qedsymbol\origqed
  \endgroup
}

\pretocmd{\proof}{\setcounter{claim}{0}}{}{}
\pretocmd{\claimproof}{\setcounter{subclaim}{0}}{}{}

\def\Pr{\mathbf{P}} 
\def\Exp{\mathbf{E}} 
\newcommand{\tr}{\textrm{tr}}
\newcommand{\sgn}{\textrm{sgn}}

\newcommand{\NN}{\mathbb{N}}
\newcommand{\RR}{\mathbb{R}}

\newcommand{\eps}{\varepsilon}
\newcommand{\Hom}{\textrm{Hom}}
\newcommand{\Aut}{\textrm{Aut}}
\newcommand{\End}{\textrm{End}}
\newcommand{\id}{\textrm{id}}

\newcommand{\eqshift}{%
  \mathrel{%
    \ooalign{%
      \hfil\raise 0.9ex\hbox{$\circ$}\hfil\cr
      $=$\cr
    }%
  }%
}

\newcommand{\cupdot}{\mathbin{\mathaccent\cdot\cup}}

\newcommand{\sym}{\mathrm{sym}}
\newcommand{\Herm}{\mathrm{Herm}}

\begin{document}

\setstretch{1.2}

\title{
A Characterization for Spectral and Trace Positivity of Matrix Words\thanks{
FG: Research supported by Czech Science Foundation Project 21-21762X. With institutional support RVO:67985807. FW: Research supported by NSF grant  DMS-2401414. 
}
\\
}
\author{Frederik Garbe\thanks{Institute of Computer Science, Czech Academy of Sciences, Czech Republic, E-mail: {\tt garbe@cs.cas.cz}} \and Fan Wei\thanks{Department of Mathematics, Duke University, 120 Science Drive, Durham, NC 27710, USA. E-mail: {\tt fan.wei@duke.edu}} }

\date{}

\maketitle

\begingroup
\makeatletter
\renewcommand{\@makefnmark}{}%
\makeatother
\endgroup

\begin{abstract}

While the global positivity of noncommutative polynomials is deeply tied to algebraic rigidity and operator algebras, with connections ranging from exact sum-of-squares certificates to the refutation of Connes’ Embedding Conjecture, the multiplicative structure governing the positivity of individual matrix products has remained poorly understood. We investigate this phenomenon at the level of individual words. Under what precise combinatorial conditions does a word in noncommuting matrix variables and their formal transposes universally evaluate to matrices with nonnegative eigenvalues or nonnegative trace, regardless of the dimension or entries of the substituted matrices?

In this paper, we provide a complete structural characterization of universally spectrally positive matrix words. Crucially, we establish that in the pure monomial setting, both spectral and trace positivity exhibit an exact rigidity phenomenon analogous to Hermitian-square representations. This contrasts sharply with the broader polynomial setting, where tracial rigidity is known to fail even for approximation. Furthermore, we establish that this structural rigidity extends seamlessly to the complex and infinite-dimensional domains. We show that the exact same purely combinatorial conditions completely characterize matrix words evaluated over complex matrices, bounded operators on arbitrary Hilbert spaces, and positive tracial states on $C^*$-algebras and von Neumann algebras.

As one application of our main theorems, we fully resolve open questions initially posed by Lieb and Pierce, and later conjectured by Hillar and Johnson, concerning the underlying algebraic structure of the summands in Bessis-Moussa-Villani (BMV) trace polynomials from quantum statistical mechanics.
We achieve these results by establishing a surprising connection to graph theory, leading to a matrix analogue of the positive graph conjecture posed by Camarena, Cs\'oka, Hubai, Lippner, and Lov\'asz. 

\end{abstract}

\newpage
\section{Introduction}\label{sec:Intro}

Determining the global nonnegativity of polynomials over $\mathbb{R}^k$ is a problem with a rich and long history, fundamentally rooted in Hilbert's 17th problem. In the classical commutative setting, the algebraic structure is notoriously delicate; global positivity does not guarantee that a polynomial can be expressed as a sum of squares. As demonstrated by Motzkin~\cite{Motz67}, there exist real polynomials that are globally nonnegative yet algebraically impossible to decompose into sums of squared polynomials, necessitating the use of rational functions instead, as established by Artin \cite{Art27}.

In stark contrast, the noncommutative world exhibits a remarkable degree of algebraic rigidity regarding matrix-level positivity, where a polynomial evaluated on matrices yields a positive semidefinite matrix. Helton \cite{Hel02} proved that if a symmetric noncommutative polynomial $p$ evaluates to a positive semidefinite matrix under every real matrix substitution of arbitrary size, then it is a sum of Hermitian squares, i.e. $p = \sum_j h_j h_j^T$. Here $p^T$ denotes the formal transpose, obtained by reversing the order of the factors in each monomial and replacing each $X_i$ by $X_i^T$ and vice versa. The polynomial $p$ is said to be \emph{symmetric} if it satisfies $p=p^T$;\footnote{Note that this is the terminology used by Helton. We will later work with a different notion of symmetry for monomials, see Definition~\ref{def:symmetric}.} for example, $p = X_1 X_1^T - 2(X_3 X_2) (X_3 X_2)^T$. However, this framework fundamentally relies on matrix-level positivity of the evaluation and therefore applies only to symmetric polynomials. In contrast, a polynomial whose evaluations merely have nonnegative eigenvalues need not be symmetric. Consequently, Helton's theorem does not cover the problem of characterizing arbitrary matrix products with universally nonnegative trace or spectrum.

Furthermore, replacing matrix-level positivity by the weaker requirement of trace nonnegativity does not lead to an equally satisfactory structural theory. In this setting, Klep and Schweighofer \cite{KS08,BDKS14} established a connection to Connes' Embedding Conjecture, a major conjecture in operator algebras. Assuming the conjecture, every symmetric polynomial that is tracially nonnegative on all norm-bounded self-adjoint matrix substitutions admits an approximation by sums of Hermitian squares and commutators; and conversely, this tracial Positivstellensatz is equivalent to Connes' embedding conjecture. However, the breakthrough $\mathrm{MIP}^*=\mathrm{RE}$ theorem \cite{JNV+21} in quantum complexity theory implied the failure of Connes' Embedding Conjecture and thereby ruled out such a fully general tracial Positivstellensatz. 
Moreover, a stronger implication of \cite{JNV+21} led to undecidability results for broad classes of positivity problems in operator algebras. 
Thus, even for symmetric polynomials and the comparatively structured notion of trace positivity, neither a general algebraic characterization nor a general decidability theory is available. There also exist concrete examples which show that exact sum-of-squares equivalences fail to capture trace positivity. For instance, the Motzkin-based polynomial $f(X,Y)=YX^4Y+XY^4X-3XY^2X+1$ is universally trace-positive over all symmetric matrix substitutions~\cite[Example 4.4]{KS08}. Evaluating $f$ on the substitutions $X = X_1 X_1^T$ and $Y = X_2 X_2^T$ therefore yields a symmetric polynomial that is universally trace-positive over all arbitrary matrix substitutions. However, it cannot be decomposed into a sum of Hermitian squares and commutators. Indeed, if such a decomposition existed, evaluating it on $1 \times 1$ matrices would cause all commutators to vanish, incorrectly implying that the scalar polynomial $x^8y^4 + x^4y^8 - 3x^4y^4 + 1$ is a classical sum of squares which contradicts Motzkin's classical result~\cite{Motz67}.

These global frameworks reveal a sharp dichotomy for general noncommutative polynomials: matrix-level positivity enjoys exact algebraic rigidity, while trace positivity fundamentally resists even approximate characterization. A natural question arises: does rigidity for spectral or trace positivity re-emerge when passing from arbitrary polynomials to monomials, i.e. individual matrix products? Formally, a matrix product is a {\it word} in matrix variables and their formal transposes, for example $W = X_1 X_2 X_1^T X_3$. Thus a word is a finite, ordered string of variables where the letters are not assumed to commute. When real square matrices of the same dimension are substituted for these variables, the word evaluates to a single matrix via standard matrix multiplication.
Motivated by the failure of general polynomial positivity theories and the highly structured nature of individual matrix products, this paper addresses the following question:
\begin{quote}\textbf{Main Question} ($\star$). Under what structural conditions does a word of matrices force real eigenvalues, nonnegative eigenvalues, or a nonnegative trace, regardless of the specific matrices forming its constituent factors, and regardless of whether those matrices are arbitrary, symmetric, or positive definite?

\smallskip

In other words, can we predict universality for the spectrum or trace of a product purely by analyzing its formal word structure, irrespective of the sizes or entries of the matrices substituted into it?
\end{quote}

Similar questions naturally arise for complex matrices, where the transpose is replaced by the conjugate transpose. Section \ref{sec:generalizations} relates the real and complex finite-dimensional spectral problems through the standard real representation of a complex matrix as a real matrix of twice the dimension; the trace result is then obtained from the same combinatorial characterization. Beyond finite dimensions, one can ask the analogous questions for bounded operators on Hilbert spaces and for positive tracial states on $C^*$-algebras. We provide complete combinatorial characterizations in all these settings.

Beyond its intrinsic structural interest, this question is closely related to problems in quantum statistical mechanics, operator algebras, and matrix analysis. 
For instance, one motivation comes from a famous classical problem in statistical physics: the Bessis-Moussa-Villani (BMV) conjecture \cite{BMV75}. Originally formulated to establish explicit error bounds for Padé approximants of partition functions in quantum mechanical systems, the BMV conjecture was later equivalently reformulated by Lieb and Seiringer~\cite{LieSeiBMV04} as a question involving the nonnegativity of traces of matrix polynomials. Specifically, they asked whether the polynomial $p(t) = \mathrm{Tr}((A+Bt)^m)$ has nonnegative coefficients for all positive definite matrices $A$ and $B$. Because the coefficient of $t^k$ in $p(t)$ is exactly the trace of the sum of all words of length $m$ containing exactly $k$ occurrences of $B$, investigation of this polynomial naturally led researchers to isolate the constituent products. Stripping away collective positivity effects induced by summation, Lieb, and independently Pierce (see references in \cite{HillarJohnson}), posed questions concerning the pure combinatorial structure of individual matrix words.
This may be viewed as a special case of our general question, for which Hillar, Johnson, and Armstrong also conjectured a complete characterization.

\begin{qu}[originally Lieb; Armstrong-Hillar \cite{ArmHillar}, Hillar-Johnson \cite{HillarJohnson}]\label{q2}
Is the trace of a given word in two letters positive for all real symmetric positive definite matrices $A$ and $B$? What is the characterization?\end{qu}
\begin{qu}[originally Lieb, Pierce; Hillar-Johnson \cite{HillarJohnson}]\label{q3}
Are all the eigenvalues of a given word in two letters positive for all real symmetric positive definite matrices $A$ and $B$? What is the characterization?
\end{qu}

As reported in \cite{HillarJohnson}, these two questions of Lieb and Pierce were initially supported by extensive numerical evidence. Multiple independent researchers tested tens of thousands of positive definite matrices and their products without finding a single negative trace. This made it tempting to conjecture an affirmative answer to Question \ref{q2}, which would in turn imply the BMV conjecture. Eventually, however, Hillar and Johnson \cite{HillarJohnson} discovered an explicit counterexample to Question \ref{q2}, specifically the word $W = (BA)^2AB$. This discovery demonstrated that the matrix substitutions required to force a negative trace can be quite rare and difficult to find. Because empirical simulations seem to be so unreliable for this problem, determining whether a word is strictly positive requires bypassing numerical tests entirely and directly classifying its exact combinatorial structure. Looking for a foundational algebraic mechanism, Hillar and Johnson essentially proposed that the correct characterization for Questions \ref{q2} and \ref{q3} should be a natural combinatorial generalization of the following well-known facts:

\begin{fact}\label{Fact:AAt}\leavevmode
\begin{enumerate}
    \item Let $A \in \mathbb{R}^{n \times n}$. Then $A$ is a symmetric positive semidefinite matrix if and only if there exists a matrix $B \in \mathbb{R}^{n \times n}$ such that $A = BB^{T}$.\label{itm:aat}
    \item For symmetric real matrices $A$ and $B$: the eigenvalues of $AB$ are real if at least one is positive semidefinite, and nonnegative if both are.\label{itm:possemi}
\end{enumerate}
\end{fact}

Building on these principles, they conjectured that universally trace-positive matrix words might be characterized via products whose factors are expressions of the form $LL^T$ for some subword $L$.
However, because Stahl~\cite{stahl} ultimately proved the BMV conjecture using complex analysis, his approach bypassed the noncommutative word expansion of the trace entirely. As a result, the algebraic mechanism originally envisioned by Lieb and Pierce was left open. How exactly does the spelling of a matrix word govern its spectrum and trace? Progress on this specific combinatorial problem stalled, as counterexamples remained difficult to construct systematically without a unifying theory and as discussed earlier, empirical evidence could be quite misleading. 
As a corollary of our main theorem answering Main Question ($\star$), we can also fully answer Questions \ref{q2} and \ref{q3}, including a full characterization, when the substituted matrices are required to be symmetric or can be arbitrary matrices with positive eigenvalues. The results also generalize to complex valued matrices with the transpose replaced with the conjugate transpose.

\

Our technique relates the semidefiniteness of a symbolic matrix product to the \emph{positive graph conjecture}~\cite[Conjecture~1]{PosGr16}, which was first posed as a problem by Lov\'asz in~\cite[Problem 22]{LoProb08}. Roughly speaking, a graph is \emph{positive} if its homomorphism density into any edge-weighted graph is always nonnegative. The mathematical bridge connecting these two seemingly distinct domains stems from the fact that the trace of a matrix product can be expanded algebraically as a sum over indices, which naturally corresponds to counting homomorphisms (or walks) from a cycle graph into an edge-weighted target graph. The conjecture asserts that a graph $G$ is positive if and only if it is \emph{symmetric}, in the sense that $G$ can be constructed by gluing two isomorphic graphs along a common independent set. (See Subsection \ref{subsec:pg} for precise definitions). To date, this beautiful conjecture remains widely open.

One of our motivations is that, if there existed a product of matrices that is non-symmetric yet always has positive eigenvalues, then the carefully designed auxiliary graph $G$ from the main proof would disprove the positive graph conjecture. However, our main theorem (Theorem \ref{thm:realreal}) shows that no counterexample can arise in this fashion. Furthermore, Theorem \ref{thm:realreal} may be regarded as a ``matrix analogue'' of the positive graph conjecture, which we prove holds universally.

\section{Main Results}

Given a set of noncommuting variables $X_1, \dots, X_\ell$, we can symbolically introduce a formal transpose $\cdot^T$ satisfying $(X_i^{T})^T=X_i$ and $(X_i X_j)^T=X_j^T X_i^T$. We consider formal expressions of the form $\prod_{j=1}^k X_{\iota(j)}^{\tau(j)}$, where $\iota:[k]\rightarrow[\ell]$ and $\tau:[k]\rightarrow\{1,T\}$.
We call such an expression a \emph{matrix word} (or simply a \emph{word}). The integer $k$ is the \emph{length} (or \emph{degree}) of the word, and $\ell$ is the number of distinct variables. By applying the transposition rules, the formal transpose $P^T$ of any word $P$ is another matrix word.
For any integer $n \in \NN$, a word $P$ defines a natural evaluation map $P: (\RR^{n\times n})^\ell \rightarrow \RR^{n\times n}$. For a tuple of matrices $(A_1, \dots, A_\ell)$, the evaluation $P(A_1, \dots, A_\ell)$ is obtained by substituting $X_i$ with $A_i$ and applying standard matrix multiplication and transposition.
We investigate the following positivity properties of matrix words:

\begin{defi}\label{def:psd}
Let $P$ be a matrix word in $\ell$ variables, where each variable $X_i$ is restricted to a permissible matrix class $\mathcal{X}_i$ (e.g., arbitrary, symmetric, or positive semidefinite real matrices). Let $\mathcal{X} = \mathcal{X}_1 \times \cdots \times \mathcal{X}_\ell$.
\begin{enumerate}
\item We say $P$ is \textbf{\emph{real-eigenvalued over $\mathcal{X}$}} if for every $n \in \NN$ and all tuples $(A_1, \dots, A_\ell) \in \mathcal{X}$ of $n\times n$-matrices, all eigenvalues of $P(A_1, \dots, A_\ell)$ are real.
\item We say $P$ is \textbf{\emph{spectrally nonnegative over $\mathcal{X}$}} if for every $n \in \NN$ and all tuples $(A_1, \dots, A_\ell) \in \mathcal{X}$ of $n\times n$-matrices, all eigenvalues of $P(A_1, \dots, A_\ell)$ are real and nonnegative.
\item We say $P$ is \textbf{\emph{trace nonnegative over $\mathcal{X}$}} if for every $n \in \NN$ and all tuples $(A_1, \dots, A_\ell) \in \mathcal{X}$ of $n\times n$-matrices, the trace $\tr(P(A_1, \dots, A_\ell))$ is nonnegative.
\end{enumerate}
\end{defi}

It is not immediately obvious whether the positivity of a given matrix product is even a decidable problem. Indeed, standard algebraic certificates are known to fail; as shown by the equivalence to the Connes Embedding Conjecture \cite{KS08}, general symmetric polynomials evaluated on symmetric matrices cannot always be approximated by sums of squares and commutators.
Therefore, the contribution of this article is to characterize word positivity via an easily verifiable combinatorial property. A natural starting point is Fact~\ref{Fact:AAt} that for any real matrix $A$, the product $AA^T$ is symmetric and has only nonnegative eigenvalues. One might be tempted to suggest a straightforward generalization of this square structure as a characterization for spectrally nonnegative matrix words. However, we must also account for the following property of matrix spectra:

\begin{fact}\label{Fact:ABBA}
Let $A, B \in \RR^{n\times n}$. Then $AB$ and $BA$ have the same eigenvalues.
\end{fact}

Note that Fact~\ref{Fact:ABBA} does not imply that any arbitrary permutation of a matrix product retains the same eigenvalues (e.g., $ABCD$ and $ACBD$ generally have different spectra). However, Fact~\ref{Fact:ABBA} does imply that eigenvalues and traces are invariant under \emph{cyclic permutation}. Thus, we require a structural definition that recognizes not only a literal square like $X_1 X_2 X_2^T X_1^T = (X_1 X_2)(X_1 X_2)^T$, but also its cyclic permutations, such as $X_2 X_2^T X_1^T X_1$. This leads to the following definition.

\begin{defi}\label{def:symmetric}\leavevmode
Let $P = \prod_{i=1}^k X_{\iota(i)}^{\tau(i)}$ be a matrix word.
\begin{enumerate}
\item We call a matrix word $Q$ \textbf{\emph{cyclically equivalent}} to $P$, denoted by $P\eqshift Q$, if there exists $j \in [k]$ such that $\prod_{i=j+1}^{j+k} X_{\iota(i)}^{\tau(i)} = Q$, where indices are considered modulo $k$.
\item We call $P$ \textbf{\emph{symmetric}} if there exists a matrix word $L$ such that $P\eqshift LL^T$. Formally, $k$ is even and there exists $j \in [k]$ such that $\prod_{i=j+1}^{j+k} X_{\iota(i)}^{\tau(i)} = LL^T$, where indices are considered modulo $k$ and $L = \prod_{i=j+1}^{j+\tfrac{k}{2}} X_{\iota(i)}^{\tau(i)}$.
\end{enumerate}
\end{defi} 
For example, $P_1 = X_1X_2 X_2^T X_1^T$ and $P_2 = X_2^T X_3 X_3^T X_2 X_1 X_1^T$ are both symmetric because $P_1 = (X_1X_2)(X_1X_2)^T$ and, up to cyclic permutation, $P_2$ can be written as $(X_1^T X_2^T X_3)(X_1^T X_2^T X_3)^T$.
We caution the reader that this combinatorial definition, requiring cyclic equivalence to a literal square, is more restrictive than the standard free algebra notion of a symmetric word. In the free algebra, one might naturally define a word as symmetric if it is cyclically equivalent to its formal transpose ($W^T \eqshift W$). However, under our definition, satisfying $P^T \eqshift P$ is necessary but not sufficient. For instance, when we later consider variables restricted to be symmetric matrices, the word $P = X_1^\sym X_2^\sym$ satisfies $P^T \eqshift P$, but it cannot be written as $LL^T$. Consequently, it is neither symmetric according to Definition~\ref{def:symmetric} nor always real-eigenvalued.

\subsection{Characterizations for Arbitrary Matrix Substitutions}
Our first main result establishes that when evaluated on arbitrary real matrices, a matrix word satisfies any of the positivity conditions from Definition~\ref{def:psd} if and only if it is structurally symmetric. If we say that a matrix word $P$ with $\ell$ variables has one of the properties from Definition~\ref{def:psd} \emph{over arbitrary real square matrices}, we mean that $\mathcal{X}_i=\bigcup_{n\in\mathbb N}\mathbb R^{n\times n}$ for every $i\in[\ell]$.

\begin{thm}\label{thm:realreal}
Let $P$ be a matrix word. The following statements are equivalent:
\begin{enumerate}
\item $P$ is symmetric (i.e., $P\eqshift LL^T$, where $L$ is a matrix word).\label{itm:realrealsym}
\item $P$ is real-eigenvalued over arbitrary real square matrices.\label{itm:realrealrealev}
\item $P$ is spectrally nonnegative over arbitrary real square matrices.
\item $P$ is trace nonnegative over arbitrary real square matrices.\label{itm:realrealtrace}
\end{enumerate}
\end{thm}
This theorem reveals a surprising collapse of the standard positivity hierarchy. Generally, possessing a nonnegative trace is a strictly weaker condition than possessing a nonnegative spectrum. However, Theorem \ref{thm:realreal} demonstrates that for individual matrix words over arbitrary matrices, these analytic properties are equivalent, and both are rigidly governed by the exact combinatorial spelling of the word.

The forward implications are straightforward algebraically: if a cyclic permutation of $P$ can be written as $LL^T$, then the spectrum of $P$ exactly coincides with that of the positive semidefinite matrix $LL^T$. This immediately guarantees that all eigenvalues of $P$ are real and nonnegative, yielding the equivalence Item 1 $\iff$ Item 3 as a direct corollary of Item 1 $\iff$ Item 2.
Conversely, the equivalence Item 1 $\iff$ Item 4 is far from an obvious corollary of the spectral results. Proving that mere trace nonnegativity forces the word to be a cyclic square requires bypassing standard algebraic certificates entirely. We derive this tracial characterization as a direct consequence of a stronger result, Theorem \ref{thm:realsym2}, which we prove later.

Ultimately, Theorem \ref{thm:realreal} shows that, in the pure monomial setting, both spectral and trace positivity exhibit an exact rigidity phenomenon analogous to Hermitian-square representations. However, we emphasize that this absolute atomic rigidity is highly fragile under addition; it is known to fail in the trace-positive setting for sums of words. For instance, there exist trace-positive noncommutative polynomials that are not cyclically equivalent to sums of Hermitian squares and commutators (see, e.g., \cite{Quar15}).

\subsection{Characterizations for Substitutions with Symmetry Restrictions}
These results can be generalized further by considering matrix words where specific variables are restricted to be symmetric matrices. Recall Fact~\ref{Fact:AAt}~\ref{itm:possemi}, which states that if $A,B\in\mathbb R^{n\times n}$ are symmetric, and $A$ or $B$ is positive semidefinite, then all eigenvalues of $AB$ are real. Consequently, the word $X_1 X_2 X_2^T$ always evaluates to a matrix with real eigenvalues provided $X_1$ is symmetric, with no restriction on $X_2$. This leads to further natural structural questions: for instance, does the word $X_1 X_2^T X_3 X_2 X_3$ always have real eigenvalues if we require $X_1$ and $X_3$ to be symmetric, but place no restrictions on $X_2$?

The following theorem provides a complete structural classification for this setting. It demonstrates that, up to cyclic permutation, the  mechanism in Fact~\ref{Fact:AAt} is not merely a sufficient condition; it is the \emph{only} algebraic mechanism that forces real eigenvalues when introducing symmetric variables. Here, if we say that a matrix word $P = \prod_{i=1}^k X_{\iota(i)}^{\tau(i)}$ with $\ell$ variables \emph{allows symmetric variables}, then we have that $\tau:[k]\rightarrow \{1,T,\sym\}$ and for each $i\in[\ell]$ we set
$$\mathcal{X}_i=\begin{cases}\bigcup_{n\in\mathbb N}\{A\in\mathbb R^{n\times n}\mid A\textrm{ symmetric }\},\textrm{ if }\iota(j)=i\textrm{ and }\tau(j)=\sym\textrm{ for some }j\in[k]\\ \bigcup_{n\in\mathbb N}\mathbb R^{n\times n},\textrm{ otherwise}\end{cases}$$
when considering properties from Definition~\ref{def:psd}.
To maintain unique syntactic factorization, we assume without loss of generality that if a variable $X_i$ is restricted to symmetric matrices, all occurrences of $X_i$ and $X_i^T$ in $P$ are formally replaced by the single symbol $X_i^\sym$. Furthermore, we extend the formal involution rules to include $(X_i^\sym)^T = X_i^\sym$.

\begin{thm}\label{thm:realsym}
Let $P$ be a matrix word allowing symmetric variables. Then $P$ is real-eigenvalued if and only if:
\begin{enumerate}
\item $P$ is symmetric (i.e., $P\eqshift LL^T$, where $L$ is a matrix word), or \label{itm:realsymsym}
\item there exist $j \in [\ell]$ and a matrix word $L$ such that $P\eqshift LL^TX_{j}^\sym$. \label{itm:realsymsymplusvar}
\end{enumerate}
\end{thm}

Note that three canonical examples that guarantee real eigenvalues naturally fit into the two cases above: $X_1X_2 (X_1X_2)^T$ falls into Case \ref{itm:realsymsym}; $X_1^\sym X_2 X_2^T$ falls into Case \ref{itm:realsymsymplusvar}; and the power $(X_1^\sym)^k$ falls into Case \ref{itm:realsymsym} if $k$ is even, and Case \ref{itm:realsymsymplusvar} if $k$ is odd.
In the context of nonnegative spectra, the classification simplifies even further, eliminating the second case entirely.

\begin{coro}\label{cor:sympossemidef}
A matrix word $P$ allowing symmetric variables is spectrally nonnegative if and only if $P$ is symmetric.
\end{coro}

Finally, we establish the  tracial result promised during our discussion of arbitrary matrices (Theorem \ref{thm:realreal}). Even when symmetric variables are introduced, the structural rigidity  remains: trace nonnegativity alone is sufficient to force the word into a symmetric square.
\begin{thm}\label{thm:realsym2}
Let $P$ be a matrix word allowing symmetric variables. Then $P$ is trace nonnegative if and only if $P$ is symmetric.
\end{thm}

\subsection{Characterizations for Positive Semidefinite Matrix Substitutions}

As an application of our main theorems, we now fully characterize when a matrix word is spectrally nonnegative or trace nonnegative over all positive semidefinite matrix substitutions.
Crucially, we demonstrate that the answers to Questions \ref{q2} and \ref{q3} emerge as direct consequences of this characterization. Although originally Lieb and Pierce and sequentially Hillar and Johnson  formulated those questions in terms of strict positivity over strictly positive definite matrices, we show that resolving the broader positive semidefinite boundary case naturally and completely resolves their strict interior conjectures.

To formalize this, we adopt the combinatorial terminology used in \cite{HillarJohnson, ArmHillar}. We say that a matrix word $P$ is \emph{palindromic} if reversing the sequence of its letters yields the identical word. Furthermore, $P$ is called \emph{nearly symmetric} if, after a cyclic permutation, it can be written as the product of two palindromes. If we say that a matrix word $P$ with $\ell$ variables has one of the properties from Definition~\ref{def:psd} \emph{over real symmetric positive semidefinite matrices}, we mean that $\mathcal{X}_i=\bigcup_{n\in\mathbb N}\{A\in\mathbb R^{n\times n}\mid A\textrm{ symmetric positive semidefinite }\}$ for every $i\in[\ell]$.

\begin{thm}\label{thm:HilJoh}
Let $P$ be a matrix word in symmetric variables. The following are equivalent:
\begin{enumerate}
\item $P$ is spectrally nonnegative over real symmetric positive semidefinite matrices.
\item $P$ is trace nonnegative over real symmetric positive semidefinite matrices.
\item $P$ is nearly symmetric.
\end{enumerate}
\end{thm}
The exact same combinatorial characterization applies when nonnegativity is replaced with strict positivity.

\begin{thm}\label{thm:HilJoh2}
Let $P$ be a matrix word in symmetric variables. The following are equivalent:
\begin{enumerate}
\item $P$ evaluates to a matrix with strictly positive eigenvalues over all real symmetric positive definite matrices.
\item $P$ has a strictly positive trace over all real symmetric positive definite matrices.
\item $P$ is nearly symmetric.
\end{enumerate}
\end{thm}

If we remove the symmetry restriction on the substituents and only require that each input matrix $X_i$ possesses a nonnegative spectrum, the ``nearly symmetric" characterization no longer works. We obtain the following theorem, demonstrating an even stricter combinatorial rigidity.

\begin{thm}\label{thm:HilJohno-sym}
Let $P$ be a matrix word. The following are equivalent:
\begin{enumerate}
\item $P$ is spectrally nonnegative (resp. evaluates to a matrix with strictly positive eigenvalues) over all real matrices with nonnegative (resp. strictly positive) spectra.\label{itm:HJsym_possp}
\item $P$ is trace nonnegative (resp. has a strictly positive trace) over all real matrices with nonnegative (resp. strictly positive) spectra.\label{itm:HJsym_postr}
\item $P$ is symmetric, or $P = (X_j^t)^k$ for some variable $X_j$ and $t\in\{1,T\}$.\label{itm:HJsym_sympower}
\end{enumerate}
\end{thm}

\subsection{Decidability, Complex Generalizations, and Operator Algebras}
Because our algebraic characterization reduces an analytic property to a pure combinatorial string-matching condition, it provides a complete resolution to the corresponding decidability problem for matrix words. This is in stark contrast to the broader polynomial setting. The breakthrough $\text{MIP}^* = \text{RE}$ theorem \cite{JNV+21} implies that there is no general algorithm to decide whether a symmetric noncommutative polynomial evaluates to a nonnegative trace over arbitrary dimensions or within a von Neumann algebra. Despite that formal undecidability for general polynomials, our result guarantees that the exact positivity of an individual matrix word can always be verified in quadratic time.

\begin{coro}\label{cor:decide}
The decision problem of whether a matrix word (including symmetric variables) $P$ is spectrally nonnegative or trace nonnegative is decidable. Moreover, if $P$ has length $k$, this can be decided in time $O(k^2)$.
\end{coro}

Before detailing the proof techniques, we note that it is natural to extend this framework beyond real matrices. In the complex setting, we can ask when a formal product composed of complex matrices and their conjugate transposes is guaranteed to have real eigenvalues, nonnegative eigenvalues, or nonnegative traces. Furthermore, these inquiries also extend into infinite-dimensional analysis, prompting analogous questions for bounded operators on Hilbert spaces and elements of $C^*$-algebras evaluated under a tracial state.

Our main results imply analogues for all of these environments. As stated below for complex matrices and operators, the structural rigidity persists: the analytic properties hold if and only if the product is formally self-adjoint or close to self-adjoint. All previous definitions are straightforwardly adjusted to the complex and operator setting, for details see Section~\ref{sec:generalizations}.

\begin{coro}\label{cor:complex}
    Let $P$ be a matrix word over complex variables and their adjoints. The following statements are equivalent:
    \begin{enumerate}
        \item $P$ is formally self-adjoint (i.e., $P\eqshift LL^*$, where $L$ is a matrix word).
        \item $P$ is real-eigenvalued over arbitrary complex square matrices.
        \item $P$ is spectrally nonnegative over arbitrary complex square matrices.
        \item $P$ is trace nonnegative over arbitrary complex square matrices.
    \end{enumerate}
\end{coro}

\begin{coro}\label{cor:HilJoh-operatorintro}
    Let $P$ be a noncommutative word over variables $\{X_1, \dots, X_\ell\}$. The following statements are equivalent:
    \begin{enumerate}
        \item For every infinite-dimensional Hilbert space $\mathcal{H}$, $P$ is spectrally nonnegative for all assignments of the variables to positive bounded linear operators from $\mathcal{H}$ to itself.
        \item For every unital $C^*$-algebra equipped with a tracial state $\tau$, $\tau(P) \ge 0$ for all assignments of the variables to positive elements in the algebra.
        \item $P$ is nearly symmetric; i.e., after a cyclic permutation, $P$ can be factored as the product of two palindromes.
    \end{enumerate}
\end{coro}

As we establish in Section~\ref{sec:generalizations}, this exact equivalence scales flawlessly to infinite dimensions. The combinatorial condition of being formally self-adjoint completely characterizes bounded operators with real or nonnegative spectra on arbitrary Hilbert spaces, as well as positive tracial states on von Neumann algebras. Similar universal structural equivalences are also obtained when variables are restricted to Hermitian or positive semidefinite elements.

\subsection{Proof Sketch and Connection to Extremal Combinatorics}

As already mentioned earlier, most results follow from Theorems~\ref{thm:realsym} and~\ref{thm:realsym2} which are proven in a very similar manner. The following is a brief sketch for the proof of Theorem~\ref{thm:realsym}.
Using the elementary mechanisms of Fact~\ref{Fact:AAt} and Fact~\ref{Fact:ABBA}, it is straightforward to show that each symmetric structural case guarantees real eigenvalues. The main difficulty of Theorem~\ref{thm:realsym} lies in proving the reverse implication: forcing the exact algebraic structure from the analytic assumption.

Given a real-eigenvalued word $P$, our main proof idea translates the algebra into extremal combinatorics. We assign a specific asymmetric edge-rooted graph $F_i$ to each variable $X_i$ and glue these graphs along the edges of an even cycle $C_{2k}$, where the orientation of the root edge is governed by $\tau(i)$. In the case of $\tau(i)=\sym$, we glue a symmetrized version of $F_i$ to the cycle. The resulting graph $G$ can then be understood as representing the squared product $P^2$. The cyclic structure of $G$, combined with the analytic fact that $P$ is real-eigenvalued, ensures that $G$ is a so-called \emph{positive graph}. If $G$ has a reflection along an axis of the cycle, then $P$ is algebraically symmetric (as illustrated in Figure~\ref{fig:cyclesym}). If no such reflection exists, we utilize a foundational result of Camarena, Cs\'oka, Hubai, Lippner, and Lov\'asz~\cite{PosGr16}, which states that all vertices with the same walk-tree in $G$ also induce a positive subgraph. This allows us to pass to a positive subgraph of $G$ consisting of a cycle with copies of $F_1$ glued to it at regular intervals. However, we demonstrate that this specific class of graphs, which resembles the appearance of a sun, is strictly \emph{not} positive, thereby reaching a contradiction.

\begin{figure}[t]
\centering
\begin{tikzpicture}[dot/.style={circle,fill,inner sep=0pt,minimum size=7pt}]
 \draw node[regular polygon,regular polygon sides=8,draw,minimum size=6cm]
 (p8){}
  foreach \X in {1,...,8} {(p8.corner \X) node[dot](n-\X){}};

\begin{scope}[every edge/.style={draw=black,very thick}]
\path [dashed] (p8.corner 2) edge (p8.corner 6);
\draw (p8.corner 3) edge["$F_{\iota(1)}^{\tau(1)}$"] (p8.corner 2);
\draw (p8.corner 2) edge["$F_{\iota(2)}^{\tau(2)}=(F_{\iota(1)}^{\tau(1)})^T$"] (p8.corner 1);
\draw (p8.corner 1) edge["$F_{\iota(3)}^{\tau(3)}=(F_{\iota(4)}^{\tau(4)})^T$"] (p8.corner 8);
\draw (p8.corner 8) edge["$F_{\iota(4)}^{\tau(4)}=(F_{\iota(3)}^{\tau(3)})^T$"] (p8.corner 7);
\draw (p8.corner 7) edge["$F_{\iota(1)}^{\tau(1)}=(F_{\iota(2)}^{\tau(2)})^T$"] (p8.corner 6);
\draw (p8.corner 6) edge["$F_{\iota(2)}^{\tau(2)}$"] (p8.corner 5);
\draw (p8.corner 5) edge["$F_{\iota(3)}^{\tau(3)}$"] (p8.corner 4);
\draw (p8.corner 4) edge["$F_{\iota(4)}^{\tau(4)}$"] (p8.corner 3);
\end{scope}
\end{tikzpicture}
\caption{The symmetry of our auxiliary graph for a positive semidefinite symbolic matrix product.}\label{fig:cyclesym}
\end{figure}
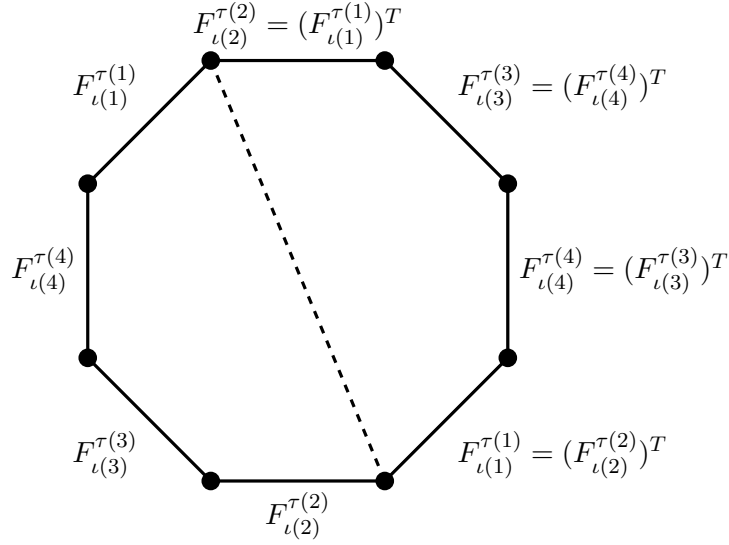

\section{Preliminaries}\label{sec:prelim}

\subsection{Graph homomorphisms and eigenvalues}

For two graphs $G$ and $H$ we call a map $\varphi:V(G)\rightarrow V(H)$ a \emph{homomorphism} if it preserves adjacency, i.e. $uv\in E(G)$ implies $\varphi(u)\varphi(v)\in E(H)$. We denote the set of homomorphisms from $G$ to $H$ by $\Hom(G,H)$ and define $\hom(G,H)\coloneqq|\Hom(G,H)|$. This notion can be extended to an \emph{edge-weighted graph} $H$, where every edge $uv\in E(H)$ gets assigned a real number $\beta_{uv}$. A homomorphism $\varphi\in\Hom(G,H)$ then also gets assigned a \emph{weight} $w_\varphi=\prod_{uv\in E(G)}\beta_{\varphi(u)\varphi(v)}$ and in this case $\hom(G,H)\coloneqq\sum_{\varphi\in\Hom(G,H)}w_\varphi$. Furthermore, we call $\varphi\in\Hom(G,G)$ an \emph{endomorphism} of $G$ and set $\End(G)\coloneqq\Hom(G,G)$. A bijection $\varphi\in\Hom(G,H)$ is called an \emph{isomorphism} if $uv\in E(G)\Leftrightarrow \varphi(u)\varphi(v)\in E(H)$. Lastly, the set of isomorphisms from $G$ to $G$ is denoted by $\Aut(G)$ and such maps are called \emph{automorphisms}. The \emph{homomorphism density} of a graph $F$ in a graph $G$ is defined by
$$t(F,G)\coloneqq\frac{\hom(F,G)}{|V(G)|^{|V(F)|}}\;.$$
We will also consider the directed version of a graph. A \emph{digraph} is a pair $G=(V,E)$ such that $E\subseteq V\times V$ and all the above definitions naturally extend to digraphs. For a digraph $G$ with vertex set $V=\{v_1,\cdots,v_n\}$, the \emph{adjacency matrix} $A\in\mathbb R^{n\times n}$  is defined by
$$A_{i,j}\coloneqq\begin{cases}1,\textrm{ if }(v_i,v_j)\in E(G)\\0,\textrm{ otherwise }.\end{cases}$$
The adjacency matrix of a simple undirected graph is a real symmetric matrix and is therefore orthogonally diagonalizable and its eigenvalues are real. However, many facts also extend to non-symmetric adjacency matrices of digraphs whose eigenvalues can be complex. The following well-known fact relating the eigenvalues of a digraph and its directed cycle count will be useful later.

\begin{prop}[\protect{\cite[Example 5.11]{LoBook}}]\label{prop:cycletr}
    Let $k\geq 2$ and $C_k$ be the directed cycle of length $k$. If $A$ is the adjacency matrix of an edge-weighted digraph $G$ with eigenvalues $\lambda_1,\cdots,\lambda_n$, then
    \[\hom(C_k,G)=\tr(A^k)=\sum_{i=1}^n\lambda_i^k\;.\]
\end{prop}

\subsection{Positive graphs}\label{subsec:pg}

A graph $G$ is called \emph{positive} if it has a positive homomorphism density in every edge-weighted graph (allowing also negative weights). See~\cite{PosGr16} and~\cite[Section 14.1]{LoBook} for a more detailed introduction. There is a conjectured characterization for positive graphs. A graph $G$ is called \emph{symmetric}, if its vertices can be partitioned into three sets $V(G)=S\cup A\cup B$ such that $S$ is an independent set, there is no edge between $A$ and $B$, and there exists an isomorphism between $G[S\cup A]$ and $G[S\cup B]$ which fixes every element of $S$. It is easy to see that every symmetric graph is positive. In~\cite{PosGr16} the authors also conjecture the reverse implication. This is one of the major conjectures concerning graph homomorphism densities and still remains completely open.

\begin{conj}[\protect{\cite[Conjecture~1]{PosGr16}}]\label{conj:posgraph}
A graph $G$ is positive if and only if it is symmetric.
\end{conj}

Theorem~\ref{thm:realreal} can therefore be seen as an analogue of Conjecture~\ref{conj:posgraph} for products of matrices. However, since there is still no proof known for Conjecture~\ref{conj:posgraph} we cannot utilize it for our proof and will instead make use of some other necessary conditions for positive graphs from~\cite{PosGr16}. The following is an easy observation.

\begin{prop}[\protect{\cite[Lemma~3]{PosGr16}}]\label{prop:posevenedges}
If $G$ is positive, then $G$ has an even number of edges.
\end{prop}

Furthermore, there is a connection between positive graphs and their walk-tree partition.

\begin{defi}
Given a rooted graph $(G,v)$, the \textbf{\emph{walk-tree}} is an infinite rooted tree denoted by $R_G(v)$. Its vertex set consists of all finite walks starting from $v$ and its root is the $0$-length walk. A walk is connected to another walk if the latter can be obtained from the former by deleting its last node. The \textbf{\emph{walk-tree partition}} $\mathcal{R}$ is the partition of $V(G)$ in which two nodes $u,v\in V(G)$ belong to the same class if and only if $R_G(u)\cong R_G(v)$ as rooted trees.
\end{defi}

Our main tool for passing to a positive subgraph of a positive graph will be given by the following lemma.

\begin{lem}[\protect{\cite[Corollary~16]{PosGr16}}]\label{lem:indwalktrcl}
Let $G=(V,E)$ be a positive graph, and let $S\subseteq V$ be the union of some classes of the corresponding walk-tree partition. Then $G[S]$ is also positive.
\end{lem}

We will make use of the following elementary observations about isomorphic walk-trees. In the following, if we denote walk-trees as isomorphic we always mean isomorphic as rooted trees.

\begin{prop}\label{prop:isoimplsameneighiso}
Let $G$ be a graph and $v,u\in V(G)$. If $R_G(u)\cong R_G(v)$, then there exists a bijection $\varphi:N_G(u)\rightarrow N_G(v)$ such that $R_G(w)\cong R_G(\varphi(w))$ for every $w\in N_G(u)$.
\end{prop}

\begin{proof}
Note that for every vertex $v\in V(G)$ and $z\in N_G(v)$ we have that the rooted subtree of $R_G(v)$ starting at $(vz)$, and denoted by $R_G(vz|v)$, is isomorphic to $R_G(z)$. Since $R_G(u)\cong R_G(v)$, there must be a bijection between $\{R_G(vz|v)\mid z\in N_G(v)\}$ and $\{R_G(uz|u)\mid z\in N_G(u)\}$ such that the corresponding pairs are isomorphic as rooted trees which by our previous observation induces the desired $\varphi$.
\end{proof}

Let $G$ be a graph and $v\in V(G)$. We call the multiset $D_G(v)\coloneqq\{d_G(w)\mid w\in N_G(v)\}$ the \emph{neighbourhood degree sequence}.

\begin{prop}\label{prop:isoimplsamendeg}
Let $G$ be a graph and $v,u\in V(G)$. If $R_G(u)\cong R_G(v)$, then $d_G(v)=d_G(u)$ and $D_G(u)=D_G(v)$.
\end{prop}

\begin{proof}
First note that $R_G(u)\cong R_G(v)$ implies that $d_{R_G(u)}((u))=d_{R_G(v)}((v))$ and $D_{R_G(u)}((u))=D_{R_G(v)}((v))$. Therefore $d_G(v)=d_{R_G(v)}((v))=d_{R_G(u)}((u))=d_G(u)$. Furthermore, for every $w\in N_G(v)$ we have $d_G(w)=d_{R_G(v)}((vw))-1$ implying $D_G(v)=\{x-1\mid x\in D_{R_G(v)}((v))\}$ as multisets. Similarly we also have $D_G(u)=\{x-1\mid x\in D_{R_G(u)}((u))\}$ and therefore $D_G(v)=D_G(u)$.
\end{proof}

\subsection{Notation}

Throughout the paper, when we use indices from a set of natural numbers $[m]=\{1,\cdots,m\}$ we often identify this index set with $\mathbb Z/m\mathbb Z$ and calculate with indices modulo $m$. It will be clear from the context and not always be mentioned explicitly.

\section{A collection of rigid graphs}\label{sec:rigidgraphs}

For our proof we need a collection of highly non-symmetric graphs whose vertices are distinguishable by their degrees or the degrees of their neighborhoods. For this we consider the well-known \emph{Erd\H os-Rényi-Gilbert random graph model} $G(n,p)$, $p\in[0,1]$, which is a probability distribution on all graphs with vertex set $[n]$ generated by including each edge independently and uniformly with probability $p$. We call a graph $G$ \emph{asymmetric} if its automorphism group is trivial, i.e. $\Aut(G)=\{\id\}$. For an even stronger property we call a graph $G$ \emph{rigid} if there is no non-trivial homomorphism from $G$ to itself, i.e. $\End(G)=\{\id\}$. It was already shown by Erd\H os and R{\'e}nyi~\cite{ErRe63} that almost all graphs are asymmetric. The same holds for rigidity. We say that a property holds for $G(n,p)$ \emph{asymptotically almost surely (a.a.s.)} if the probability that $G\sim G(n,p)$ satisfies the property tends to $1$ as $n \to \infty$.

\begin{thm}[\protect{\cite[Theorem~4.7]{HeNe04}}]\label{thm:rigid}
    Let $G\sim G(n,\tfrac{1}{2})$. Then a.a.s. $G$ is rigid.
\end{thm}

We also want to distinguish vertices by the degrees of their neighbors. For that we use a result of Babai, Erd\H os and Selkow~\cite{BaerSe80} that the degrees of the highest degree vertices in $G(n,\tfrac{1}{2})$ are all distinct.

\begin{thm}[\protect{\cite[Theorem~1.2 and Theorem~3.6]{BaerSe80}}]\label{thm:ndegrsequ}
    Let $r=\lceil 3\log(n)/\log(2)\rceil$ and $G\sim G(n,\tfrac{1}{2})$. Denote by $d_1\geq d_2\geq\cdots\geq d_n$ the degrees of $G$ and let $S\subseteq V(G)$ be the set of vertices corresponding to the degrees $d_1,\cdots,d_{r}$. Then a.a.s. for every $1\leq i<j\leq r+2$ we have $d_i>d_j$ and for every $\{u,v\}\subseteq V(G)\setminus S$ that $N_G(v)\cap S\neq N_G(u)\cap S$.
\end{thm}

Using those two theorems and sampling random graphs of different sizes we can obtain a collection of graphs whose vertices are easily distinguishable. We call a pair $(F,e)$ where $F$ is a graph and $e=(u,w)$ an orientation of an edge of $F$ an \emph{edge-rooted graph}.

\begin{prop}\label{prop:rigidfamily}
For every $\ell,n_0\in\NN$ there exists a family $\{(F_1,(a_1,b_1)),\cdots,(F_{\ell},(a_\ell,b_\ell))\}$ of edge-rooted graphs, each on $n_i\geq n_0$, $i\in[\ell]$, vertices, such that the following properties hold.
\begin{enumerate}[label=(\arabic*)]
    \item $F_{i}$ is rigid for every $i\in[\ell]$.\label{itm:rigidf}
    \item For every $i\in[\ell]$, for every pair of vertices $\{u,v\}\subseteq V(F_i)$ there exist $x,y,z\in N_{F_i}(u)\cap N_{F_i}(v)$ such that $xy,xz,yz\in E(F_i)$.\label{itm:connectivity}
    \item $N_{F_i}(a_i)\setminus N_{F_i}(b_i)\neq\emptyset\neq N_{F_i}(b_i)\setminus N_{F_i}(a_i)$ for every $i\in [\ell]$.\label{itm:notcommonneigh}
    \item $d_{F_i}(a_i)<d_{F_i}(b_i)$ for every $i\in[\ell]$.\label{itm:rootdeg}
    \item $d_{F_i}(u)<d_{F_j}(v)$ for all $i,j\in[\ell]$ with $i<j$ and all $u\in V(F_i)$ and $v\in V(F_j)$.\label{itm:diffdeg}
        \item \label{itm:sumdeg}     If $i,j,i',j'\in[\ell]$ with $\{i,j\}\neq\{i',j'\}$, then the set of values:
        \begin{align*}
        \mathcal{D}_{i,j} = \{d_{F_i}(u_1)+ d_{F_i}(u_1') + d_{F_j}(v_1) + d_{F_j}(v_1') - 2, \ \ d_{F_i}(u_1)+ d_{F_i}(u_1') + d_{F_j}(v_1) - 1,  \\
        d_{F_i}(u_1)+d_{F_j}(v_1): u_1, u_1'\in V(F_i),v_1, v_1'\in V(F_j) \}
        \end{align*}
        is disjoint from the set of values
          \begin{align*}
            \mathcal{D}_{i',j'} =\{d_{F_{i'}}(u_2)+ d_{F_{i'}}(u_2')+d_{F_{j'}}(v_2) + d_{F_{j'}}(v_2') - 2, \ \ d_{F_{i'}}(u_2)+ d_{F_{i'}}(u_2')+d_{F_{j'}}(v_2) - 1, \\ d_{F_{i'}}(u_2)+d_{F_{j'}}(v_2) : u_2, u_2'\in V(F_{i'}),v_2, v_2'\in V(F_{j'})\}.
              \end{align*}
   \item \label{itm:sumdeg2}
        Additionally, the set of values
        \[ \bigcup_{i,j \in [\ell]} \mathcal{D}_{i,j} \] is disjoint from
        \[ \{d_{F_t}(w): t\in[\ell] ,w\in F_t\}\;.\]

    \item For every $i\in[\ell]$ there exists $S_i\subseteq V(F_i)\setminus\{a_i,b_i\}$ such that $d_{F_i}(u)>d_{F_i}(v)$ for every $u\in S_i$ and $v\in V(F_i)\setminus S_i$, and $d_{F_i}(u)\neq d_{F_i}(v)$ for every $u,v\in S_i$ with $u\neq v$. Moreover, $N_{F_i}(u)\cap S_i\neq N_{F_i}(v)\cap S_i$ for every $u,v\in V(F_i)\setminus S_i$ with $u\neq v$.
    \label{itm:degsequneighb}
\end{enumerate}
\end{prop}
\begin{proof}
We choose $\eps$ sufficiently small with respect to $1/\ell$ and $n$ sufficiently large with respect to $n_0$ and $1/\varepsilon$. Then for every $i\in[\ell]$ we set $n_i\coloneqq 9^in$ and sample $F_i\sim G(n_i,\tfrac{1}{2})$. Note that each of the following properties holds for every $i\in[\ell]$.
\begin{enumerate}[label=(\roman*)]
    \item By Theorem~\ref{thm:rigid} $F_i$ is rigid with probability $1-o(1)$.\label{itm:pfrigid}
    \item For every pair $\{u,v\}\subseteq V(F_i)$ let $X_{u,v}'=|N_{F_i}(u)\cap N_{F_i}(v)|$. Then $\Exp[X_{u,v}']=(\tfrac{1}{2})^{2}(n_i-2)$ and a standard Chernoff bound shows that $\Pr(X_{u,v}'\geq n_i/8)\geq 1-\exp(-\Omega(n_i))$. Furthermore, it is well-known that the probability that $G(n_i/8,\tfrac{1}{2})$ does not contain a triangle is at most $\exp(-\Omega(n_i))$, e.g. \cite[Theorem 3.9]{JanLuRu2000RG}. Therefore by the union bound the probability that all pairs of vertices contain at least one triangle in their common neighbourhood is at least $1-n_i^2\cdot 2\exp(-\Omega(n_i))\geq 1-o(1)$.\label{itm:pfpaths}
    \item For every $u,v\in V(F_i)$ with $u\neq v$ let $X_{u,v}$ be the number of vertices $w\in V(F_i)$ such that $\{u,w\}\in E(F_i)$, but $\{v,w\}\notin E(F_i)$. Then $\Exp[X_{u,v}]= \tfrac{1}{2}(1-\tfrac{1}{2})(n_i-2)$ and again a standard Chernoff bound shows that $\Pr(X_{u,v}\geq n_i/8)\geq 1-\exp(-\Omega(n_i))$ and by the union bound with probability at least $1-o(1)$ every pair of vertices $\{u,v\}\subseteq V(F_i)$ has $N_{F_i}(u)\setminus N_{F_i}(v)\neq\emptyset\neq N_{F_i}(v)\setminus N_{F_i}(u)$.\label{itm:pfdiffneighbourh}
    \item For every $v\in V(F_i)$ the expected number of neighbors of $v$ is $(n_i-1)/2$ and a standard Chernoff bound shows that with probability at least $1-\exp(-\Omega(n_i))$ we have $d_{F_i}(v)=(\tfrac{1}{2}\pm\eps)n_i$. A union bound over all vertices ensures that $d_{F_i}(v)=(\tfrac{1}{2}\pm\eps)n_i$ for every $v\in V(F_i)$ with probability $1-o(1)$.\label{itm:pfdegrees}
    \item By Theorem~\ref{thm:ndegrsequ} we have with probability $1-o(1)$ that for every $i\in[\ell]$ and $r_i=\left\lceil \tfrac{3\log(n_i)}{\log(2)}\right\rceil$ the following holds. Let $v_1^{(i)},\cdots,v_{r_i+2}^{(i)}$ be the vertices of the highest degrees and set $S_i=\{v_1^{(i)},\cdots,v_{r_i}^{(i)}\}$. Then $d_{F_i}(v_s^{(i)})>d_{F_i}(v_t^{(i)})$ for every $1\leq s<t\leq r_{i}+2$ and $N_{F_i}(u)\cap S_i\neq N_{F_i}(v)\cap S_i$ for every $\{u,v\}\subseteq V(F_i)\setminus S_i$.\label{itm:pfSiprofiles}
\end{enumerate}
Thus by a union bound over constantly many events we have that all the above properties hold simultaneously with probability $1-o(1)$ and we can fix such an outcome $\{F_1,\cdots,F_\ell\}$ of the random selection. We set $b_i\coloneqq v_{r_i+1}^{(i)}$ and choose an arbitrary $a_i\in N_{F_i}(b_i)\setminus S_i$ for every $i\in[\ell]$ which is possible, since $N_{F_i}(b_i)\setminus S_i\neq\emptyset$ by $|S_i|\leq 3\log(n_i)<(1-\eps)n_i/2\leq d_{F_i}(b_i)$. We show that the collection of edge-rooted graphs $\{(F_1,(a_1,b_1)),\cdots,(F_\ell,(a_\ell,b_\ell))\}$ has the claimed properties.

The properties \ref{itm:pfrigid}, \ref{itm:pfpaths}, and \ref{itm:pfdiffneighbourh} immediately imply \ref{itm:rigidf}, \ref{itm:connectivity}, and \ref{itm:notcommonneigh}. Property~\ref{itm:pfSiprofiles} together with $b_i=v_{r_i+1}^{(i)}$ and $a_i=v_{j}^{(i)}$, for some $j>r_i+1$, for every $i\in[\ell]$ implies \ref{itm:rootdeg}. By~\ref{itm:pfdegrees} we have $d_{F_i}(u)<(\tfrac{1}{2}+\eps)n_i=(\tfrac{1}{2}+\eps)9^in<(\tfrac{1}{2}-\eps)9^jn=(\tfrac{1}{2}-\eps)n_j=d_{F_j}(v)$ for every $i,j\in[\ell]$ with $i<j$, $u\in V(F_i)$, and $v\in V(F_j)$ implying~\ref{itm:diffdeg}.

For~\ref{itm:sumdeg}, using the conclusion of \ref{itm:pfdegrees}, $d_{F_i}(u_1)+ d_{F_i}(u_1') + d_{F_j}(v_1) + d_{F_j}(v_1') -2 = d_{F_{i'}}(u_2)+ d_{F_{i'}}(u_2') + d_{F_{j'}}(v_2) + d_{F_{j'}}(v_2') - 2$ implies that
$9^{i}+ 9^i + 9^{j} + 9^j=9^{i'}+ 9^{i'} + 9^{j'} + 9^{j'} \pm  16 \eps 9^{\ell}$. Since $8\eps 9^{\ell}<1$, this implies $9^i+9^j=9^{i'}+9^{j'}$ which in turn implies that $\{i,j\}=\{i',j'\}$. Comparing all other sums from $\mathcal{D}_{i,j}$ and $\mathcal{D}_{i',j'}$ with an almost identical argument leads to $\mathcal{D}_{i,j} \cap \mathcal{D}_{i',j'}=\emptyset$, for $\{i,j\} \neq \{i', j'\}$.

~\ref{itm:sumdeg2} is proved by a similar argument, as $d_{F_i}(u_1)+ d_{F_i}(u_1') + d_{F_j}(v_1) + d_{F_j}(v_1')-2=d_{F_t}(w)$ implies $9^i + 9^i + 9^j + 9^j =9^{t}$ leading to the contradiction $0\equiv 1$ modulo $2$. The other two cases are checked analogously.

Lastly, \ref{itm:degsequneighb} immediately follows from~\ref{itm:pfSiprofiles} while also noticing that we chose $\{a_i,b_i\}$ disjoint from $S_i$ for every $i\in[\ell]$.
\end{proof}

\section{Set-up of the Main Proof}\label{sec:set_up}

The idea for proving Theorems~\ref{thm:realsym} and~\ref{thm:realsym2} is to use the collection of graphs from Section \ref{sec:rigidgraphs} to construct another graph $G$ representing $P$, and to show that if the formal product $P$ is real-eigenvalued, then $G$ must be a positive graph. We then show that if $P$ fails to satisfy the conclusion of Theorem \ref{thm:realsym}, then the graph $G$ cannot be positive, yielding a contradiction.

\paragraph{Construction of $G$.}
Recall from the previous section that we call a pair $(F,(a,b))$, where $F$ is a graph and $(a,b)$ an orientation of an edge of $F$, an {\it edge-rooted graph}. For two edge-rooted graphs $(F_1,(a_1,b_1))$ and $(F_2,(a_2,b_2))$ we write $(F_1,(a_1,b_1))\cong (F_2,(a_2,b_2))$ if there exists an isomorphism $\varphi\in\Hom(F_1,F_2)$ with $\varphi(a_1)=a_2$ and $\varphi(b_1)=b_2$.
We often write $F$ instead of $(F,(a,b))$ when the root is clear from the context.
We additionally define $F^T=(F,(b,a))$.
Furthermore, we define the \emph{symmetrization} $F^{\sym}$ of $F$ to be the following graph. Consider two disjoint copies of $(F, (a_1, b_1))$ and $F^T = (F, (b_2, a_2))$, and identify $a_1$ and $b_2$ as well as $b_1$ and $a_2$. Then $F^\sym$ is a graph with roots $a = a_1 = b_2, b = a_2 = b_1$ which are symmetric in the sense that $F^\sym$ has an automorphism which maps $a$ to $b$ and $b$ to $a$. We identify $F^\sym$ with the edge-rooted graph $(F^\sym, (a, b))$.

Consider pairwise disjoint edge-rooted graphs $(F_1,(a_1,b_1)),\cdots,(F_k,(a_k,b_k))$. We construct a new graph \[G(F_1,\cdots,F_k)\] by considering a cycle $C_k=v_1v_2\cdots v_k$ and gluing the $F_i$ along $C_k$ by identifying each of the triples $b_{i-1}$, $a_i$ and $v_i$ where indices are considered modulo $k$. See Figure~\ref{fig:cycleglue} for an illustration. 

We also define a \emph{doubling} of $G$ which for convenience we denote by 
\begin{equation}G^2(F_1,\cdots,F_k)\coloneqq G(F_1,\cdots,F_k,F_{k+1},\cdots,F_{2k}), \label{eq:Gsquare} \end{equation} where $F_{k+i}$ is a disjoint copy of $F_i$, $i\in[k]$.

We will also make use of the following observation about the symmetrization of certain rigid graphs.

\begin{prop}\label{prop:automsym}
Let $(F,(a,b))$ be an edge-rooted graph. If $F$ is rigid and such that every two vertices of $F$ contain a copy of $K_3$ in their common neighbourhood, then $\Hom(F^{\sym},F^{\sym})=\{\id, \sigma\}$ and $\sigma$ is an automorphism with $\sigma(a)=b$ and $\sigma(b)=a$.
\end{prop}

\begin{proof}
Let $\varphi\in\Hom(F^{\sym},F^{\sym})$. We denote the two copies of $(F,(a,b))$ which are glued together in $F^{\sym}$ by $(F_1,(a,b))$ and $(F_2,(b,a))$. First note that $\varphi(V(F_1))\subseteq V(F_i)$ for some $i\in[2]$. For the sake of contradiction suppose that there exist $x\in (\varphi(V(F_1))\cap V(F_1))\setminus\{a,b\}$ and $y\in (\varphi(V(F_1))\cap V(F_2))\setminus\{a,b\}$. Then there exist $x'\in\varphi^{-1}(x)\cap V(F_1)$ and $y'\in\varphi^{-1}(y)\cap V(F_1)$ and by assumption there exists a copy of $K_3$ in $N_{F_1}(x')\cap N_{F_1}(y')$. Therefore $x$ and $y$ must also contain a copy of $K_3$ in $N_{F^\sym}(x)\cap N_{F^\sym}(y)$, but $N_{F^\sym}(x)\cap N_{F^\sym}(y)\subseteq\{a,b\}$. Analogously we have $\varphi(F_2)\subseteq V(F_j)$ for some $j\in[2]$. Now, note that $i\neq j$, as otherwise, if we denote by $\{\psi_1\}=\Hom(F,F_1)$ and $\{\psi_2\}=\Hom(F,F_2)$ the unique isomorphisms, we get that $\alpha\coloneqq\psi_i^{-1}\circ\varphi\circ\psi_1$ and $\beta\coloneqq\psi_i^{-1}\circ\varphi\circ\psi_2$ are both in $\Hom(F,F)$, but $\alpha(a)=\psi_i^{-1}(\varphi(a))\neq\psi_i^{-1}(\varphi(b))=\beta(a)$ contradicting the rigidity of $F$. Finally, since $F$ is rigid, we have that $\varphi|_{V(F_i)}$ is the unique isomorphism for each $i\in[2]$. Hence the only two possibilities for $\varphi$ are to isomorphically map $F_1$ to $F_1$ and $F_2$ to $F_2$ or $F_1$ to $F_2$ and $F_2$ to $F_1$ corresponding to the two possibilities of the assertion.
\end{proof}

\subsection{Relating that \texorpdfstring{$P$}{P} is real-eigenvalued to the positivity of \texorpdfstring{$G$}{G}.}\label{sec:positivityG}

Our first observation is that  there is a close relationship between the properties of the eigenvalues of the matrix product $P$ and the positivity of the graph $G(F_1,\dots, F_k)$ or the graph $G^2(F_1, \dots, F_k)$, if defined via a real-eigenvalued symbolic matrix product.

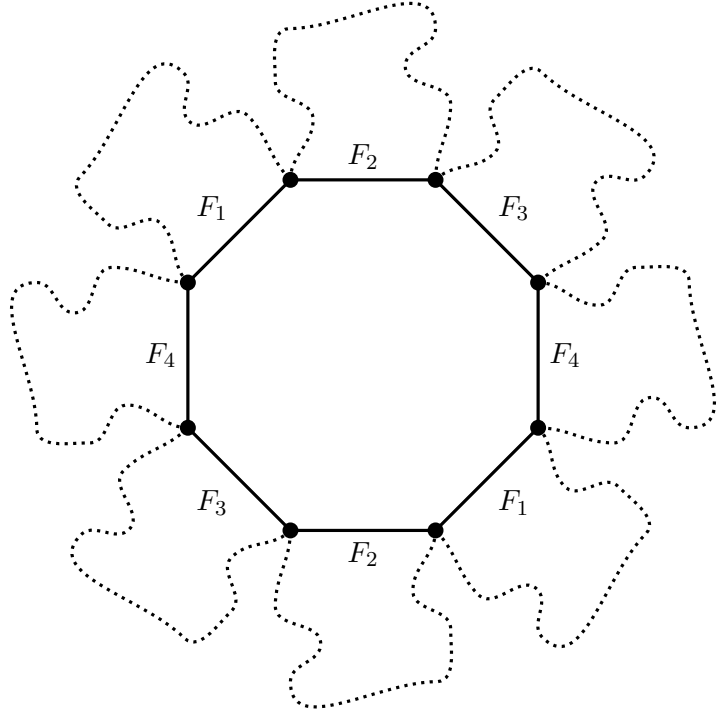
\begin{figure}[t]
\centering
\begin{tikzpicture}[dot/.style={circle,fill,inner sep=0pt,minimum size=6pt}]

 \draw node[regular polygon,regular polygon sides=8,draw,minimum size=5cm]
 (p8){} foreach \X in {1,...,8} {(p8.corner \X) node[dot](n-\X){}};

\begin{scope}[every edge/.style={draw=black,very thick}]
\draw (p8.corner 3) edge["$F_1$"] (p8.corner 2);
\draw (p8.corner 2) edge["$F_2$"] (p8.corner 1);
\draw (p8.corner 1) edge["$F_3$"] (p8.corner 8);
\draw (p8.corner 8) edge["$F_4$"] (p8.corner 7);
\draw (p8.corner 7) edge["$F_1$"] (p8.corner 6);
\draw (p8.corner 6) edge["$F_2$"] (p8.corner 5);
\draw (p8.corner 5) edge["$F_3$"] (p8.corner 4);
\draw (p8.corner 4) edge["$F_4$"] (p8.corner 3);
\end{scope}

\draw[very thick,dotted] (p8.corner 2) to[out=90,in=270] ($(p8.corner 2)+(0.3,0.8)$) to[out=90,in=270] ($(p8.corner 2)+(-0.2,1.3)$) to[out=90,in=180] ($(p8.corner 2)+(0,2)$) to[out=0,in=90] ($(p8.corner 1)+(0,2)$) to[out=270,in=90] ($(p8.corner 1)+(-0.4,1.5)$) to[out=270,in=90] ($(p8.corner 1)+(0.2,1)$) to[out=270,in=90] (p8.corner 1);

\draw[very thick,dotted,rotate=315] (p8.corner 1) to[out=90,in=270] ($(p8.corner 1)+(0.3,0.8)$) to[out=90,in=270] ($(p8.corner 1)+(-0.2,1.3)$) to[out=90,in=180] ($(p8.corner 1)+(0,2)$) to[out=0,in=90] ($(p8.corner 8)+(0,2)$) to[out=270,in=90] ($(p8.corner 8)+(-0.4,1.5)$) to[out=270,in=90] ($(p8.corner 8)+(0.2,1)$) to[out=270,in=90] (p8.corner 8);

\draw[very thick,dotted,rotate=270] (p8.corner 8) to[out=90,in=270] ($(p8.corner 8)+(0.3,0.8)$) to[out=90,in=270] ($(p8.corner 8)+(-0.2,1.3)$) to[out=90,in=180] ($(p8.corner 8)+(0,2)$) to[out=0,in=90] ($(p8.corner 7)+(0,2)$) to[out=270,in=90] ($(p8.corner 7)+(-0.4,1.5)$) to[out=270,in=90] ($(p8.corner 7)+(0.2,1)$) to[out=270,in=90] (p8.corner 7);

\draw[very thick,dotted,rotate=225] (p8.corner 7) to[out=90,in=270] ($(p8.corner 7)+(0.3,0.8)$) to[out=90,in=270] ($(p8.corner 7)+(-0.2,1.3)$) to[out=90,in=180] ($(p8.corner 7)+(0,2)$) to[out=0,in=90] ($(p8.corner 6)+(0,2)$) to[out=270,in=90] ($(p8.corner 6)+(-0.4,1.5)$) to[out=270,in=90] ($(p8.corner 6)+(0.2,1)$) to[out=270,in=90] (p8.corner 6);

\draw[very thick,dotted,rotate=180] (p8.corner 6) to[out=90,in=270] ($(p8.corner 6)+(0.3,0.8)$) to[out=90,in=270] ($(p8.corner 6)+(-0.2,1.3)$) to[out=90,in=180] ($(p8.corner 6)+(0,2)$) to[out=0,in=90] ($(p8.corner 5)+(0,2)$) to[out=270,in=90] ($(p8.corner 5)+(-0.4,1.5)$) to[out=270,in=90] ($(p8.corner 5)+(0.2,1)$) to[out=270,in=90] (p8.corner 5);

\draw[very thick,dotted,rotate=135] (p8.corner 5) to[out=90,in=270] ($(p8.corner 5)+(0.3,0.8)$) to[out=90,in=270] ($(p8.corner 5)+(-0.2,1.3)$) to[out=90,in=180] ($(p8.corner 5)+(0,2)$) to[out=0,in=90] ($(p8.corner 4)+(0,2)$) to[out=270,in=90] ($(p8.corner 4)+(-0.4,1.5)$) to[out=270,in=90] ($(p8.corner 4)+(0.2,1)$) to[out=270,in=90] (p8.corner 4);

\draw[very thick,dotted,rotate=90] (p8.corner 4) to[out=90,in=270] ($(p8.corner 4)+(0.3,0.8)$) to[out=90,in=270] ($(p8.corner 4)+(-0.2,1.3)$) to[out=90,in=180] ($(p8.corner 4)+(0,2)$) to[out=0,in=90] ($(p8.corner 3)+(0,2)$) to[out=270,in=90] ($(p8.corner 3)+(-0.4,1.5)$) to[out=270,in=90] ($(p8.corner 3)+(0.2,1)$) to[out=270,in=90] (p8.corner 3);

\draw[very thick,dotted,rotate=45] (p8.corner 3) to[out=90,in=270] ($(p8.corner 3)+(0.3,0.8)$) to[out=90,in=270] ($(p8.corner 3)+(-0.2,1.3)$) to[out=90,in=180] ($(p8.corner 3)+(0,2)$) to[out=0,in=90] ($(p8.corner 2)+(0,2)$) to[out=270,in=90] ($(p8.corner 2)+(-0.4,1.5)$) to[out=270,in=90] ($(p8.corner 2)+(0.2,1)$) to[out=270,in=90] (p8.corner 2);

\end{tikzpicture}
\caption{The construction of $G^2(F_1,F_2,F_3,F_4)$.}\label{fig:cycleglue}
\end{figure}

\begin{prop}\label{pro:cyclegluepositive}

Let $P = X_{\iota(1)}^{\tau(1)}\cdots X_{\iota(k)}^{\tau(k)}$ be a  symbolic matrix product with $\ell$ variables and of degree $k$ where each $\tau(i) \in \{1, T, \sym\}$. Let $(F_i,(a_i,b_i))$, $i\in[\ell]$, be edge-rooted graphs.

\begin{enumerate}
    \item If $P$ is real-eigenvalued, then the graph $G=G^2(F_{\iota(1)}^{\tau(1)},\cdots,F_{\iota(k)}^{\tau(k)})$ is positive.\label{itm:realeg}
    \item If $P$ is trace nonnegative or spectrally nonnegative, then the graph $G=G(F_{\iota(1)}^{\tau(1)},\cdots,F_{\iota(k)}^{\tau(k)})$ is positive.\label{itm:tracepos}
\end{enumerate}
\end{prop}

\begin{proof}
We first prove the first item, i.e., we assume $P$ always has real eigenvalues.  
Let $H$ be an edge-weighted graph with vertex set $[n]$. For every $i\in[\ell]$ we define matrices $S_i,S_i^T\in\mathbb R^{n\times n}$ by setting for every $x,y\in[n]$
$$S_i(x,y)\coloneqq\sum_{\substack{\varphi\in\Hom(F_i,H)\\\varphi(a_i)=x,\varphi(b_i)=y}}w_\varphi\;\textrm{ and }\;S_i^T(x,y)\coloneqq\sum_{\substack{\varphi\in\Hom(F_i,H)\\\varphi(a_i)=y,\varphi(b_i)=x}}w_\varphi\;.$$
Similarly, define
$$S_i^\sym(x,y)\coloneqq\sum_{\substack{\varphi\in\Hom(F_i^\sym,H)\\\varphi(a_i)=x,\varphi(b_i)=y}}w_\varphi$$
and note that $S_i^\sym(x,y)=S_i^\sym(y,x)$ for all $x,y\in[n]$ because of the existence of $\psi\in \Aut(F_i^\sym)$ with $\psi(a_i)=b_i$.

Furthermore, we define the matrix $U\in\mathbb R^{n\times n}$ by setting for every $x_1,x_{k+1}\in[n]$
\[U(x_1,x_{k+1})\coloneqq\sum_{x_2,\cdots,x_{k}\in[n]}\prod_{i=1}^kS^{\tau(i)}_{\iota(i)}(x_i,x_{i+1})\;.\]
Note that then by the definition of the matrix product $U=P(S_1^{t_1},\cdots,S_{\ell}^{t_\ell})$, where for every $i\in[\ell]$ we set
$$t_i=\begin{cases}1,\textrm{ if }\tau(j)\neq\sym\textrm{ for every }j\in[k]\textrm{ with }\iota(j)=i\\ \sym,\textrm{ if }\tau(j)=\sym\textrm{ for some }j\in[k]\textrm{ with }\iota(j)=i\;.\end{cases}$$
Hence the input for $P$ is well-defined and all eigenvalues $\lambda_i$, $i\in[n]$, of $U$ are real, as $P$ is real-eigenvalued, and therefore $\tr(U^2)=\sum_{i\in[n]}\lambda_i^2\geq 0$. Finally,
\begin{align*}
\hom(G,H)&=\sum_{x_1,\cdots,x_{2k}\in[n]}\prod_{i\in[k]}S_{\iota(i)}^{\tau(i)}(x_i,x_{i+1})S_{\iota(i)}^{\tau(i)}(x_{k+i},x_{k+i+1})\\
&=\sum_{x_1,x_{k+1}\in[n]}\left(\sum_{x_2,\cdots,x_{k}\in[n]}\prod_{i\in[k]}S_{\iota(i)}^{\tau(i)}(x_i,x_{i+1})\right)\\
&\;\cdot\left(\sum_{x_{k+2},\cdots,x_{2k}\in[n]}\prod_{i\in[k]}S_{\iota(i)}^{\tau(i)}(x_{k+i},x_{k+i+1})\right)\\
&=\sum_{x_1,x_{k+1}\in[n]}U(x_1,x_{k+1})U(x_{k+1},x_1)\\
&=\tr(U^2)\geq 0\;.
\end{align*}

The second claim is proved similarly. In this case, by assumption we have $\tr(U) \geq 0$ or all eigenvalues of $U$ are nonnegative. In either case, $\tr(U) \geq 0$. Furthermore, here we consider $G=G(F_{\iota(1)}^{\tau(1)},\cdots,F_{\iota(k)}^{\tau(k)})$. A similar calculation as above gives 
\begin{align*}
\hom(G,H)&=\sum_{\substack{x_1,\dots,x_k \in [n] \\ x_{k+1} = x_1}}\prod_{i\in[k]}S_{\iota(i)}^{\tau(i)}(x_i,x_{i+1})\\
&=\sum_{x_1\in[n]}\left(\sum_{x_2,\cdots,x_{k}\in[n]}\prod_{i\in[k]}S_{\iota(i)}^{\tau(i)}(x_i,x_{i+1})\right)
=\sum_{x_1\in[n]}U(x_1,x_{1})
=\tr(U)\geq 0\;.
\end{align*}
\end{proof}

\subsection{A family of non-positive graphs}\label{sec:sunnotpos}
To show that $G$ is not a positive graph when $P$ fails to satisfy the assertions of Theorem \ref{thm:realsym}, we first prove that certain families of graphs (which will later appear as induced subgraphs of $G$) are themselves not positive. More specifically, we show that if we glue copies of a rigid graph in regular intervals along a cycle, then the resulting graph is not positive. In our main proof we will apply this to graphs given by Proposition~\ref{prop:rigidfamily}.

\begin{defi}
    Let $(F,(a,b))$ be an edge-rooted graph and $k,\ell\in\mathbb N$. Then we define the \textbf{\emph{sun graph}} $SG((F,(a,b)),k,\ell)$ to be the graph obtained from $C_{k(\ell+1)}=\{v_1,\cdots,v_{k(\ell+1)}\}$ by gluing $k$ copies of $F$ to the cycle such that $(a,b)$ is identified with $(v_{i(\ell+1)+1},v_{i(\ell+1)+2})$ for every $i\in[k]$. We denote the $i$'th copy glued on the cycle by $F^i$, its root vertices by $(a_i,b_i)=(v_{i(\ell+1)+1},v_{i(\ell+1)+2})$, and the path following this copy by $P^i=(v_{i(\ell+1)+2},\cdots,v_{(i+1)(\ell+1)+1})$, $i\in[k]$.
\end{defi}

Figure~\ref{fig:sungraph} illustrates an example of a sun graph.

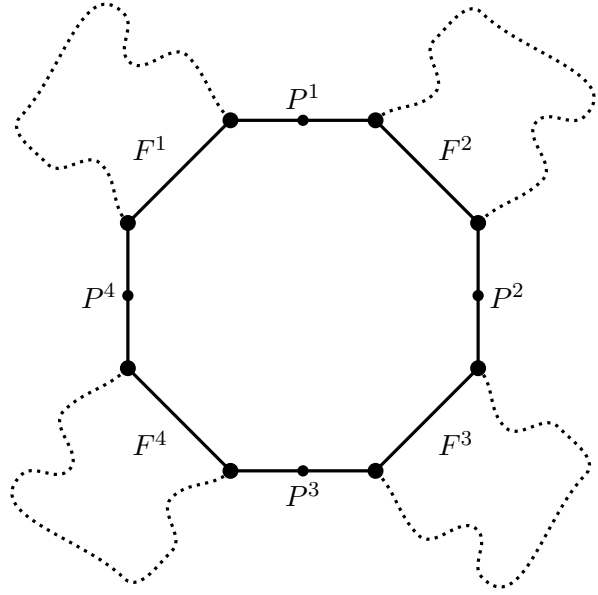
\begin{figure}[h]
\centering
\begin{tikzpicture}[dot/.style={circle,fill,inner sep=0pt,minimum size=6pt}]

 \draw node[regular polygon,regular polygon sides=8,draw,minimum size=5cm]
 (p8){} foreach \X in {1,...,8} {(p8.corner \X) node[dot](n-\X){}};

\begin{scope}[every edge/.style={draw=black,very thick}]
\draw (p8.corner 3) edge["$F^1$"] (p8.corner 2);
\draw (p8.corner 2) edge["$P^1$"] (p8.corner 1);
\draw (p8.corner 1) edge["$F^2$"] (p8.corner 8);
\draw (p8.corner 8) edge["$P^2$"] (p8.corner 7);
\draw (p8.corner 7) edge["$F^3$"] (p8.corner 6);
\draw (p8.corner 6) edge["$P^3$"] (p8.corner 5);
\draw (p8.corner 5) edge["$F^4$"] (p8.corner 4);
\draw (p8.corner 4) edge["$P^4$"] (p8.corner 3);
\end{scope}

\coordinate (mid1) at ($(p8.corner 2)!0.5!(p8.corner 1)$);
\node[circle, fill=black, inner sep=1.5pt] at (mid1) {};
\coordinate (mid2) at ($(p8.corner 8)!0.5!(p8.corner 7)$);
\node[circle, fill=black, inner sep=1.5pt] at (mid2) {};
\coordinate (mid3) at ($(p8.corner 6)!0.5!(p8.corner 5)$);
\node[circle, fill=black, inner sep=1.5pt] at (mid3) {};
\coordinate (mid4) at ($(p8.corner 4)!0.5!(p8.corner 3)$);
\node[circle, fill=black, inner sep=1.5pt] at (mid4) {};

\draw[very thick,dotted,rotate=315] (p8.corner 1) to[out=90,in=270] ($(p8.corner 1)+(0.3,0.8)$) to[out=90,in=270] ($(p8.corner 1)+(-0.2,1.3)$) to[out=90,in=180] ($(p8.corner 1)+(0,2)$) to[out=0,in=90] ($(p8.corner 8)+(0,2)$) to[out=270,in=90] ($(p8.corner 8)+(-0.4,1.5)$) to[out=270,in=90] ($(p8.corner 8)+(0.2,1)$) to[out=270,in=90] (p8.corner 8);

\draw[very thick,dotted,rotate=225] (p8.corner 7) to[out=90,in=270] ($(p8.corner 7)+(0.3,0.8)$) to[out=90,in=270] ($(p8.corner 7)+(-0.2,1.3)$) to[out=90,in=180] ($(p8.corner 7)+(0,2)$) to[out=0,in=90] ($(p8.corner 6)+(0,2)$) to[out=270,in=90] ($(p8.corner 6)+(-0.4,1.5)$) to[out=270,in=90] ($(p8.corner 6)+(0.2,1)$) to[out=270,in=90] (p8.corner 6);

\draw[very thick,dotted,rotate=135] (p8.corner 5) to[out=90,in=270] ($(p8.corner 5)+(0.3,0.8)$) to[out=90,in=270] ($(p8.corner 5)+(-0.2,1.3)$) to[out=90,in=180] ($(p8.corner 5)+(0,2)$) to[out=0,in=90] ($(p8.corner 4)+(0,2)$) to[out=270,in=90] ($(p8.corner 4)+(-0.4,1.5)$) to[out=270,in=90] ($(p8.corner 4)+(0.2,1)$) to[out=270,in=90] (p8.corner 4);

\draw[very thick,dotted,rotate=45] (p8.corner 3) to[out=90,in=270] ($(p8.corner 3)+(0.3,0.8)$) to[out=90,in=270] ($(p8.corner 3)+(-0.2,1.3)$) to[out=90,in=180] ($(p8.corner 3)+(0,2)$) to[out=0,in=90] ($(p8.corner 2)+(0,2)$) to[out=270,in=90] ($(p8.corner 2)+(-0.4,1.5)$) to[out=270,in=90] ($(p8.corner 2)+(0.2,1)$) to[out=270,in=90] (p8.corner 2);

\end{tikzpicture}
\caption{The sun graph $G(F,k,\ell)$ with $k=4$, $\ell=2$, and rigid $F$.}\label{fig:sungraph}
\end{figure}

\begin{lem}\label{lem:sunnotpos}
Suppose $\ell\geq 2, k\geq 1$. Let $(F,(a,b))$ be an edge-rooted connected graph for which it holds that
\begin{enumerate}[label=(\roman*)]
    \item $F$ has no cut-vertex,\label{itm:cutvtx}
    \item every edge of $F$ is contained in a copy of $K_3$, and \label{itm:K3}
    \item there is at most one non-trivial homomorphism $\varphi\in\Hom(F,F)$ and this is an automorphism for which $\varphi(a)=b$ and $\varphi(b)=a$.\label{itm:almrig}
\end{enumerate}
Then $G=SG((F,(a,b)),k,\ell)$ is not positive if one of the following holds:
\begin{enumerate}[label=(\arabic*)]
    \item $\ell$ is odd;\label{itm:lodd}
    \item $\ell$ is even, $k\geq 2$, and $F$ is rigid.\label{itm:levenrigid}
\end{enumerate}
\end{lem}

\begin{proof}
We first assume that $k\geq 2$ and show under this assumption the statements for~\ref{itm:lodd} and~\ref{itm:levenrigid}.
We introduce $\varepsilon>0$ sufficiently small such that
\begin{equation}\label{eq:epsbound}
\varepsilon<\min\left\{\frac{1}{(17\Delta(F))^{k\ell}},
\sin(\tfrac{\pi}{k})\right\}\;.
\end{equation}
We construct an edge weighted graph $H$ such that $\hom(G,H)<0$. For this we first set $\theta\coloneqq\pi/k$ and consider the matrix
\begin{equation}\label{eq:rotationm}
M\coloneqq\begin{bmatrix}
\cos(\theta) & -\sin(\theta)\\
\sin(\theta) & \cos(\theta)
\end{bmatrix}\;.
\end{equation}
Since $M$ is a rotation matrix with angle $\theta$ the eigenvalues are $e^{\pm i\pi/k}$ and therefore
\begin{equation}\label{eq:trneg}
\tr(M^{k})=(e^{i\pi/k})^{k}+(e^{-i\pi/k})^{k}=-2\;.
\end{equation}
Let $D$ be the edge-weighted digraph with loops whose adjacency matrix is $M$. In particular $D$ has two vertices and has two (weighted) self-loops.

Furthermore, we denote by $F_{(a,b),\ell}=(F,(a,b))+P_{\ell}$ the graph which consists of $F$ and a path $(w_1,\cdots,w_{\ell+1})$ of length $\ell$ glued to it by identifying $b$ and $w_1$. Now replace each directed edge $(x,y)$, $x,y\in[2]$, in $D$ by its own copy $F^{x,y}_{(a_{x,y},b_{x,y}),\ell}$ of $F_{(a,b),\ell}$ by removing $(x,y)$ and identifying $a_{x,y}$ and $x$ as well as $w_{\ell+1}$ and $y$, see Figure~\ref{fig:DH}. We call this copy $F_{(a,b),\ell}^{x,y}$ and we denote the copy of $F$ in $F_{(a,b),\ell}^{x,y}$ by $F^{x,y}$ and the copy of $P_\ell$ by $P^{x,y}$.
We call a vertex in $\{a_{x,y}\mid x,y\in[2]\}$ \emph{of type $A$} and $\{b_{x,y}\mid x,y\in[2]\}$ \emph{of type $B$}.
Furthermore, each edge in the copy of $F$ in $F_{(a,b),\ell}^{x,y}$ is assigned the weight $\varepsilon$ and each edge of the copy of $P_\ell$ is assigned the weight $(|M[x,y]|/\varepsilon^{e(F)})^{1/\ell}$ except the first one which is assigned $-(|M[x,y]|/\varepsilon^{e(F)})^{1/\ell}$ if $M[x,y]<0$. After replacing each directed edge we denote the resulting simple edge-weighted graph by $H$.
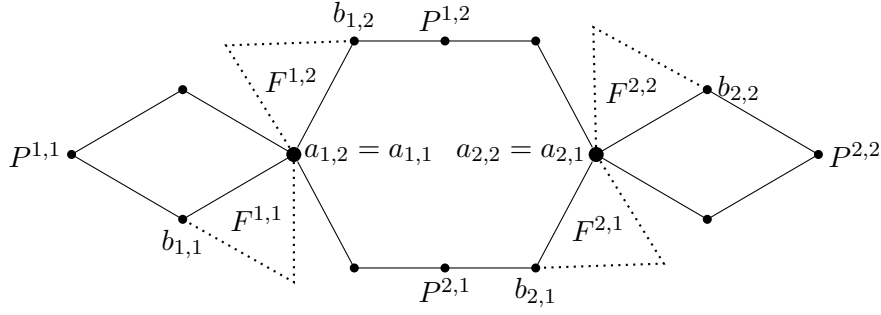
\begin{figure}[t]
\centering
\begin{tikzpicture}[>=stealth]

\node[circle, fill, inner sep=2pt] (A) at (0,0) {};
\node[circle, fill, inner sep=2pt] (B) at (4,0) {};

\coordinate (AL_right) at (0,0);
\coordinate (AL_bottom) at (-1.47,-0.857);
\coordinate (AL_left)   at (-2.94,0);
\coordinate (AL_top)    at (-1.47,0.857);
\draw[-] (AL_right) -- (AL_bottom) -- (AL_left) -- (AL_top) -- (AL_right);
\foreach \p in {AL_bottom,AL_left,AL_top} \node[circle,fill,inner sep=1.2pt] at (\p) {};

\node[right] at ($(AL_right)$) {$a_{1,2}=a_{1,1}$};
\node[below] at ($(AL_bottom)$) {$b_{1,1}$};

\coordinate (ALapex) at (0.002,-1.688);
\draw[dotted, thick] (AL_right) -- (ALapex) -- (AL_bottom);
\node at ($0.333*(AL_right)+0.333*(AL_bottom)+0.333*(ALapex)$) {$F^{1,1}$};

\node[left] at ($(AL_left)$) {$P^{1,1}$};

\coordinate (BL_left)  at (4,0);
\coordinate (BL_top)   at (5.47,0.857);
\coordinate (BL_right) at (6.94,0);
\coordinate (BL_bottom) at (5.47,-0.857);
\draw[-] (BL_left) -- (BL_top) -- (BL_right) -- (BL_bottom) -- (BL_left);
\foreach \p in {BL_top,BL_right,BL_bottom} \node[circle,fill,inner sep=1.2pt] at (\p) {};

\node[left] at ($(BL_left)$) {$a_{2,2}=a_{2,1}$};
\node[right] at ($(BL_top)$) {$b_{2,2}$};

\coordinate (BLapex) at (3.962,1.682);
\draw[dotted, thick] (BL_left) -- (BLapex) -- (BL_top);
\node at ($0.333*(BL_left)+0.333*(BL_top)+0.333*(BLapex)$) {$F^{2,2}$};

\node[right] at ($(BL_right)$) {$P^{2,2}$};

\coordinate (AB1) at (0.8,1.5);
\coordinate (AB2) at (2,1.5);
\coordinate (AB3) at (3.2,1.5);
\coordinate (ABapex) at (-0.896,1.442);
\draw[dotted, thick] (A) -- (ABapex) -- (AB1);
\node at ($0.333*(A)+0.333*(AB1)+0.333*(ABapex)$) {$F^{1,2}$};

\node[above] at ($(AB2)$) {$P^{1,2}$};

\node[above] at ($(AB1)$) {$b_{1,2}$};

\draw[-] (A) -- (AB1) -- (AB2) -- (AB3) -- (B);
\foreach \p in {AB1,AB2,AB3} \node[circle,fill,inner sep=1.2pt] at (\p) {};

\coordinate (BA1) at (3.2,-1.5);
\coordinate (BA2) at (2,-1.5);
\coordinate (BA3) at (0.8,-1.5);
\coordinate (BAapex) at (4.895,-1.442);
\draw[dotted, thick] (B) -- (BAapex) -- (BA1);
\node at ($0.333*(B)+0.333*(BA1)+0.333*(BAapex)$) {$F^{2,1}$};

\node[below] at ($(BA2)$) {$P^{2,1}$};

\node[below] at ($(BA1)$) {$b_{2,1}$};

\draw[-] (B) -- (BA1) -- (BA2) -- (BA3) -- (A);
\foreach \p in {BA1,BA2,BA3} \node[circle,fill,inner sep=1.2pt] at (\p) {};

\end{tikzpicture}
\caption{The graph $H$ after the edges of $D$ are replaced by copies of $F+P_\ell$ and $F$ is possibly symmetric in $a$ and $b$. It contains $4$ copies of $F$, denoted by $F^{x,y}$ and each followed by a path of length $\ell$ denoted by $P^{x,y}$.}\label{fig:DH}
\end{figure}

Now we call a homomorphism $\varphi\in \Hom(G,H)$ \emph{aligned} if for every $j\in[k]$ we have that
\begin{enumerate}
\item $\varphi(V(F^j))\subseteq V(F^{x,y})$ for some $x,y\in[2]$ and
\item $\varphi(V(P^j))\subseteq V(P^{x',y'})$ for some $x',y'\in[2]$.
\end{enumerate}
Otherwise we call $\varphi$ \emph{misaligned}. We then set $\Hom_{\textrm{ali}}\coloneqq\{\varphi\in \Hom(G,H)\mid \varphi\textrm{ is aligned}\}$ and $\Hom_{\textrm{mis}}\coloneqq\Hom(G,H)\setminus \Hom_{\textrm{ali}}$.

\begin{claim}\label{clm:FmapstoFinH}
    Let $\varphi\in \Hom(F,H)$. Then $\varphi$ maps $F$ isomorphically onto $F^{x,y}$ for some $x,y\in[2]$.
\end{claim}

\begin{claimproof}
Note that any edge of $F$ can only be mapped to an edge of some $F^{i,j}$ in $H$, as those are the only edges in $H$ which are contained in a triangle. Since $F$ has at least one edge, we can fix some $F^{x,y}$, $x,y\in[2]$, containing the image of this edge under $\varphi$. Then by the initial observation and, since $H[\varphi(V(F))]$ is connected, as $F$ is connected, we have that $\varphi(V(F))\subseteq V(F^{x,1})\cup V(F^{x,2})$. Suppose that $\varphi(V(F))\cap V(F^{x,1})\setminus\{x\}\neq\emptyset$ and $\varphi(V(F))\cap V(F^{x,2})\setminus\{x\}\neq\emptyset$. Then $\varphi^{-1}(x)$ is a vertex cut of $F$, but, since $F$ is connected and has no cut vertex, it follows that $|\varphi^{-1}(x)|\geq 2$. We define $\varphi'\in \Hom(F,F)$ by mapping every $v\in V(F)$ to the copy of $\varphi(v)$ in $F$. Then $\varphi'$ is a non-injective homomorphism contradicting~\ref{itm:almrig}.
Therefore $\varphi(V(F))\subseteq V(F^{x,y})$ and again by our assumption on $F$ we have that $\varphi$ is an isomorphism.
\end{claimproof}

We can now crudely upper bound the number of misaligned homomorphisms. By Claim~\ref{clm:FmapstoFinH} we know that under $\varphi\in \Hom(G,H)$ every copy of $F$ in $G$ must be mapped to a copy of $F$ in $H$, in particular, to some $F^{x,y}$ for some $x,y \in [2]$, and by~\ref{itm:almrig} there are at most two possible ways how this can be done. Hence there are at most $8^{k}$ many choices for this. Once such a choice is fixed, then for the images of the vertices of the attached path segments there are at most $\Delta(H)^{k\ell}\leq(2\Delta(F)+2)^{k\ell}$ many choices. This results into
\begin{equation}\label{eq:numberirreghoms}
    |\Hom_{\textrm{mis}}|\leq |\Hom(G,H)|\leq 8^{k}\cdot (2\Delta(F)+2)^{k\ell}\leq (17\Delta(F))^{k\ell}\;.
\end{equation}

We now consider the aligned homomorphisms.

\begin{claim}\label{clm:countalignedhom}
\begin{enumerate}[label=(\arabic*)]
\item If $F$ is rigid, then there exists a bijection $f:\Hom(C_{k},D)\rightarrow \Hom_{\textrm{ali}}$ such that $w_\psi=w_{f(\psi)}$ for every $\psi\in\Hom(C_{k},D)$.\label{itm:alighomrig}
\item If $F$ is not rigid (and $\ell$ is odd), then there exist two injections $f_1:\Hom(C_{k},D)\rightarrow \Hom_{\textrm{ali}}$ and $f_2:\Hom(C_{k},D)\rightarrow \Hom_{\textrm{ali}}$ such that $\textrm{im}(f_1)\cupdot \textrm{im}(f_2)=\Hom_{\textrm{ali}}$ and $w_\psi=w_{f_1(\psi)}=w_{f_2(\psi)}$ for every $\psi\in\Hom(C_{k},D)$.\label{itm:alighomnotrig}
\end{enumerate}
\end{claim}

\begin{claimproof}
First note that we can identify $\Hom(C_{k},D)$ with the set $[2]^{k}$ in the natural way. For~\ref{itm:alighomrig} first consider any $i=(i_1,\cdots,i_{k})\in [2]^{k}$ representing $\varphi\in\Hom(C_{k},D)$. Mapping $F^j+P^j$ onto $F_{(a,b),\ell}^{i_j,i_{j+1}}$ for every $j\in[k]$ defines a map $f(\varphi)\in\Hom_{\textrm{ali}}$ with
$$w_{f(\varphi)}=\prod_{j=1}^{k}\left(\prod_{e\in E(F)}\varepsilon\cdot\sgn(M[i_j,i_{j+1}])\cdot\prod_{e\in E(P_{\ell})}(|M[i_j,i_{j+1}]|/\varepsilon^{e(F)})^{1/\ell}\right)=\prod_{j=1}^{k}M[i_j,i_{j+1}]=w_{\varphi}\;.$$
It immediately follows from the definition that $f$ is injective, it remains to show that $f$ is also surjective. Given $\varphi\in\Hom_{\textrm{ali}}$ by Claim~\ref{clm:FmapstoFinH} we have for every $j\in[k]$ that $\varphi$ maps $F^j$ isomorphically onto some $F^{x_j,y_j}$, $x_j,y_j\in[2]$, and, since $F$ is rigid, we have that $\varphi(b_j)=b_{x_j,y_j}$ forcing that $P^j$ is mapped isomorphically onto $P^{x_j,y_j}$ and $F^{j+1}$ onto $F^{y_j,1}$ or $F^{y_j,2}$. Considering this for every $j\in[k]$ leads to a unique $i=(i_1,\cdots,i_{k})\in [2]^{k}$ representing $\psi\in\Hom(C_{k},D)$ with $f(\psi)=\varphi$.

For~\ref{itm:alighomnotrig} note that $\ell$ is odd by our assumption and $F$ is not rigid. We define $f_1$ exactly as $f$ in the proof of~\ref{itm:alighomrig}. However, since $F$ is not rigid, we also have another way of defining such a function. Given $i=(i_1,\cdots,i_{k})\in [2]^{k}$ representing $\varphi\in\Hom(C_{k},D)$. For every $j\in[k]$, we map $F^j$ onto $F^{i_{j},i_{j-1}}$ using the unique non-trivial automorphism provided by~\ref{itm:almrig} such that $f_2(\varphi)(a_j)=b_{i_j,i_{j-1}}$ and $f_2(\varphi)(b_j)=a_{i_j,i_{j-1}}$. Then we map $P^j$ isomorphically onto the reverse ordering of $P^{i_{j+1},i_j}$ so that this embedded segment ends with $f_2(\varphi)(v_{(j+1)(\ell+1)+1})=b_{i_{j+1},i_j}$. Since we end in the vertex $b_{i_{j+1},i_j}$ we can inductively repeat this for every $j\in[k]$ until we embed all of $G$ defining the injective map $f_2$. Since $f_2(\varphi)$ does not map $F^1+P^1$ onto $F_{(a,b),\ell}^{x,y}$ for any $x,y\in[2]$ and $\varphi\in\Hom(C_{k},D)$ we have $\textrm{im}(f_1)\cap \textrm{im}(f_2)=\emptyset$.

To show that $\textrm{im}(f_1)\cup \textrm{im}(f_2)=\Hom_{\textrm{ali}}$ let $\varphi\in\Hom_{\textrm{ali}}$. Since $\varphi$ is aligned and by Claim~\ref{clm:FmapstoFinH} we have for every $j\in[k]$ that $\varphi$ maps $F^j$ isomorphically onto some $F^{x,y}$, $x,y\in[2]$. By~\ref{itm:almrig} we either have that $\varphi(a_j)$ is of type $A$ and $\varphi(b_j)$ is of type $B$ or $\varphi(a_j)$ is of type $B$ and $\varphi(b_j)$ is of type $A$. In the former case we have $\varphi(V(P^j))\subseteq V(P^{x,y})$, since $\varphi$ is aligned. Furthermore, since the endpoints of $P^j$ are root vertices of copies of $F$ in $G$ their images must also be root vertices of copies of $F$ in $H$. Since $\ell$ is odd, this implies that $\varphi$ maps $P^j$ isomorphically onto $P^{x,y}$. Iterating this argument shows that $\varphi\in\textrm{im}(f_1)$. Now, in the latter case we have $\varphi(V(P^j))\subseteq V(P^{x,x})$ or $\varphi(V(P^j))\subseteq V(P^{x+1,x})$, since $\varphi$ is aligned. Arguing as before shows that then $P^j$ is mapped isomorphically onto $P^{x,x}$ or $P^{x+1,x}$, respectively, but in reverse ordering of the vertices. Iterating this leads to an image given by $f_2$.
\end{claimproof}

Thus by Claim~\ref{clm:countalignedhom} and Proposition~\ref{prop:cycletr} there exists $c\in\{1,2\}$ such that
\begin{equation}\label{eq:sumreghoms}
\sum_{\varphi\in \Hom_{\textrm{ali}}}w_{\varphi}=c\cdot\hom(C_{k},D)=c\cdot \tr(M^{k})\overset{\eqref{eq:trneg}}{=}-2c\;.
\end{equation}

Lastly, we bound the weight of a misaligned homomorphism.

\begin{claim}\label{clm:weightirreghoms}
    For every $\varphi\in \Hom_{\textrm{mis}}$ we have that $|w_{\varphi}|\leq\varepsilon$.
\end{claim}

\begin{claimproof}
Let $\varphi\in \Hom_{\textrm{mis}}$. Note that by Claim~\ref{clm:FmapstoFinH} every copy of $F$ in $G$ is mapped onto a copy of $F$ in $H$. Furthermore, consider a copy $F'_{\ell}=F'+P'_{\ell}$ of $F+P_{\ell}$ in $G$ and let $a'$ and $w_{\ell+1}'$ be the corresponding copies of $a$ and $w_{\ell+1}$. Let $\tilde{F}$ be the copy of $F$ in $H$ such that $\varphi(V(F'))=V(\tilde{F})$ and $\tilde{F}_{\ell}=\tilde{F}+\tilde{P}_{\ell}$ be the copy of $F+P_{\ell}$ in $H$ containing $\tilde{F}$ and denote the vertices corresponding to $a$ and $w_{\ell+1}$ by $\tilde{a}$ and $\tilde{w}_{\ell+1}$. We then have that $\varphi(a')=\tilde{a}$, as $F'$ is mapped isomorphically onto $\tilde{F}$, and $\varphi(w_{\ell+1}')\in\{\tilde{a},\tilde{w}_{\ell+1}\}$, because $w_{\ell+1}'$ must be mapped to a copy of $a$ (recall that by construction every copy of $w_{\ell+1}$ is identified with some copy of $a$). If $\varphi(w_{\ell+1}')=\tilde{w}_{\ell+1}$, then $F'_\ell$ is mapped isomorphically onto $\tilde{F}_\ell$. If $\varphi(w_{\ell+1}')=\tilde{a}$, then at least one edge of $P_\ell'$ is mapped on an edge of $\tilde{F}_\ell$. Since $\varphi$ is misaligned, the latter case must happen at least once. Thus all $ke(F)$ many edges of copies of $F$ in $G$ are mapped to edges of weight $\varepsilon$ in $H$, at least one edge of a copy of $P_{\ell}$ in $G$ is mapped to an edge of weight $\varepsilon$ in $H$ and at most $k\ell-1$ many of the remaining edges of $G$ are mapped to edges which do not belong to copies of $F$ and have weight at most $(\|M\|_\infty/\eps^{e(F)})^{1/\ell}$. Thus the weight of $\varphi$ is at most
\[|w(\varphi)|\leq (\eps^{e(F)})^{k}\cdot\eps\cdot\left(\left(\frac{\|M\|_\infty}{\eps^{e(F)}}\right)^{1/\ell}\right)^{k\ell-1}=\eps^{1+e(F)/\ell}\cdot \|M\|_\infty^{(k\ell-1)/\ell}\leq\eps\;,\]
and since $\|M\|_\infty\leq 1$ and $\varepsilon<1$ by \eqref{eq:epsbound}, this is at most $\varepsilon$.
\end{claimproof}

Thus by combining our observations \eqref{eq:sumreghoms}, \eqref{eq:numberirreghoms}, and Claim~\ref{clm:weightirreghoms} we conclude
\[\hom(G,H)=\sum_{\varphi\in \Hom_{\textrm{ali}}}w_{\varphi}+\sum_{\varphi\in \Hom_{\textrm{mis}}}w_{\varphi}\leq -1 +  (17\Delta(F))^{k\ell}\cdot\varepsilon<0\;,\]
which holds by \eqref{eq:epsbound}, i.e. our choice of $\varepsilon$.

It is left to show~\ref{itm:lodd} for the case $k=1$. Note that in this case $G$ consists just of one copy of $F$ whose root vertices are connected by an additional path $P$ of odd length $\ell\geq 3$.
For this we define the weighted graph $H$ to be a copy of $G$ where we assign the weight $\eps$ to every edge in $F$, $-1$ to one of the edges in $P$ and $1$ to every other edge. Since the following reasoning is very similar to the one we have seen before we keep the proof brief. First, observe that in every homomorphism from $G$ to $H$ the copy of $F$ in $G$ maps isomorphically onto the copy of $F$ in $H$. Now, the dominant maps are exactly those where the homomorphic image of the path traverses the whole path $P$ in $H$ from one root endpoint to the other. These maps cross the unique negative path edge exactly once. Every other homomorphism sends at least one path edge into the copy of $F$, hence its weight gains an extra factor $\varepsilon$.
This shows that for $\eps$ sufficiently small $\hom(G,H)$ is dominated by the isomorphisms from $G$ to $H$ which all have negative weight and therefore $\hom(G,H)<0$.
\end{proof}

\subsection{Symmetries in edge-colored oriented cycles}\label{sec:reflrot}

In our main proof we will make use of certain edge-colored oriented cycles as auxiliary graphs. In this section we show two statements which guarantee certain symmetries for those graphs, in particular rotation and reflection. As usual all indices $i\in [n]$ are considered modulo $n$.
Let $V=(v_0,\cdots,v_{n-1})$ be an ordered $n$-tuple of distinct vertices, $E=\{e_i=\{v_i,v_{i+1}\}\mid i\in[n]\}$ a set of undirected edges, and $A$ a set of colors.
Given a map $c:E\rightarrow \mathbb F_3\times A$, we then call the pair $C_n=(V,c)$ an \emph{edge-colored oriented cycle} of order $n$. For $e\in E$ we call $c(e)$ the directed color of $e$, with $c(e)_1$ the orientation of $e$, and $c(e)_2$ the undirected color of $e$. Furthermore, for convenience we define $\overline{c(e)}=(-c(e)_1,c(e)_2)$. Note that $\overline{\overline{c(e)}}=c(e)$.

Lastly, let $\varphi:\{v_0,\cdots,v_{n-1}\}\rightarrow \{v_0,\cdots,v_{n-1}\}$ be a bijection such that the following holds for every $i\in[n]$. There exists $j\in[n]$ such that $\{\varphi({v_i}),\varphi({v_{i+1}})\}=\{v_j,v_{j+1}\}$ as a set and
$$c(\{\varphi(v_i),\varphi(v_{i+1})\})=\begin{cases}c(\{v_j,v_{j+1}\})\textrm{, if }(\varphi(v_i),\varphi(v_{i+1}))=(v_j,v_{j+1})\\\overline{c(\{v_j,v_{j+1}\})}\textrm{, if }(\varphi(v_i),\varphi(v_{i+1}))=(v_{j+1},v_j)\end{cases}\;.$$
We then call $\varphi$ a \emph{color respecting automorphism} of $C_n$.

\begin{prop}\label{prop:cyclerotation}
Let $C_n=((v_0,\cdots,v_{n-1}),c)$ be an edge-colored oriented cycle of order $n$. Suppose there exist $i,j\in[n]$ with $i\neq j$ such that for all $s\in[n]$ one of the two conditions holds which might depend on the value of $s$:
\begin{align}
\label{eq:condneigh}
\begin{split}
&c(e_{i+s-1})=c(e_{j+s-1})\textrm{ and } c(e_{i+s})=c(e_{j+s})\textrm{ or }\\
&c(e_{i+s-1})=\overline{c(e_{j+s})}\textrm{ and } c(e_{i+s})=\overline{c(e_{j+s-1})}\;.
\end{split}
\end{align}
Then one of the following two cases holds.
\begin{enumerate}
    \item $\varphi(v_{x})=v_{x+j-i}$, $x\in[n]$, defines a color respecting automorphism of $C_n$.
    \item $\varphi(v_{x})=v_{x+2}$, $x\in[n]$,  defines a color respecting automorphism of $C_n$. In this case, for any edge $e_x$ in $C_n$, it satisfies $c(e_x)_1 = 0$.
\end{enumerate}
\end{prop}

\begin{proof}
W.l.o.g. we may assume that $i=0$ and $1\leq j\leq n-1$. If for every $s\in [n]$ we have that $c(e_{s})=c(e_{j+s})$, then the assertion follows. So suppose for the sake of contradiction that there is $s_0\in[n]$ such that $c(e_{s_0})\neq c(e_{j+s_0})$. Then, by \eqref{eq:condneigh} for $s=s_0$, it holds that $c(e_{s_0})=\overline{c(e_{j+s_0-1})}$ and $c(e_{s_0-1})=\overline{c(e_{j+s_0})}$. In particular we have $c(e_{s_0})\neq\overline{c(e_{s_0-1})}$ as we have assumed $c(e_{s_0})\neq c(e_{j+s_0})$.\\

\begin{claim}\label{clm:norotcolalt}
    For every $k\in[n]$ we have that    $c(e_{s_0+2k})=c(e_{s_0})=\overline{c(e_{j+s_0+2k+1})}$ and $c(e_{s_0+2k+1})=c(e_{s_0-1})=\overline{c(e_{j+s_0+2k})}$.
\end{claim}

\begin{claimproof}
We show the statement by induction. First observe that it holds for $k=0$. Indeed by \eqref{eq:condneigh} for $s=s_0+1$ and $c(e_{s_0})\neq c(e_{j+s_0})$ we have $c(e_{s_0})=\overline{c(e_{j+s_0+1})}$ and $c(e_{s_0+1})=\overline{c(e_{j+s_0})}=c(e_{s_0-1})$.
Now suppose it holds for $k\in [n]$. This means $c(e_{s_0+2k+1})=c(e_{s_0-1})$ and $c(e_{j+s_0+2k+1})=\overline{c(e_{s_0})}$. By \eqref{eq:condneigh} for $s=s_0+2(k+1)$ and $c(e_{s_0})\neq\overline{c(e_{s_0-1})}$ we have $c(e_{s_0+2(k+1)})=\overline{c(e_{j+s_0+2k+1})}=c(e_{s_0})$ and $c(e_{j+s_0+2(k+1)})=\overline{c(e_{s_0+2k+1})}=\overline{c(e_{s_0-1})}$. Finally considering \eqref{eq:condneigh} with $s=s_0+2(k+1)+1$ and again $c(e_{s_0})\neq\overline{c(e_{s_0-1})}$ we have $c(e_{s_0+2(k+1)+1})=c(e_{s_0-1})$ and $c(e_{j+s_0+2(k+1)+1})=\overline{c(e_{s_0})}$.
\end{claimproof}

If $j$ is odd Claim~\ref{clm:norotcolalt} implies $c(e_{s_0+j-1})=c(e_{s_0})$. However our assumption   from the beginning is that $c(e_{s_0})=\overline{c(e_{j+s_0-1})}$. Therefore $c(e_{s_0}) =\overline{c(e_{s_0})}$. By a similar argument, Claim~\ref{clm:norotcolalt} implies  $c(e_{s_0+2k+1})=c(e_{s_0-1})=\overline{c(e_{1+s_0+2k})}$ for any integer $k$, and therefore $c(e_{s_0-1}) =\overline{c(e_{s_0-1})}$. Therefore the coloring of $C_n$ is alternating between $c(e_{s_0})$ and $c(e_{s_0-1})$, and for both of these
$c(e_{s_0-1})_1 = c(e_{s_0})_1 = 0$.

If $j$ is even Claim~\ref{clm:norotcolalt} implies that $c(e_{s_0+j-1})=c(e_{s_0-1})$ which with the observations from the beginning shows that $c(e_{s_0})=\overline{c(e_{j+s_0-1})}=\overline{c(e_{s_0-1})}$ contradicting $c(e_{s_0})\neq\overline{c(e_{s_0-1})}$.
\end{proof}

\begin{prop}\label{prop:cyclereflection}
Let $C_n=((v_0,\cdots,v_{n-1}),c)$ be an edge-colored oriented cycle of order $n$. Suppose there exists $i\in[n]$ such that for every $s\in[n]$ we have
\begin{align}
\label{eq:condneigh2}
\begin{split}
&c(e_{s-1})=c(e_{i-s-1})\textrm{ and } c(e_{s})=c(e_{i-s})\textrm{ or }\\
&c(e_{s-1})=\overline{c(e_{i-s})}\textrm{ and } c(e_{s})=\overline{c(e_{i-s-1})}\;.
\end{split}
\end{align}
If there exist no directed colors $c_0$ and $c_1$ such that $c(e_{2k})=c_0$ and $c(e_{2k+1})=c_1$ for every $k\in[n]$, then $\varphi(v_{x})=v_{i-x}$ defines a color respecting automorphism of $C_n$.
\end{prop}

\begin{proof}
If for every $s\in [n]$ we have that $c(e_{s})=\overline{c(e_{i-s-1})}$, then $\varphi$ is a color respecting automorphism of $C_n$, as for every $x\in[n]$ we have $(\varphi(v_x),\varphi(v_{x+1}))=(v_{i-x},v_{i-x-1})$ and
$$c(\{\varphi(v_x),\varphi(v_{x+1})\})=c(\{v_{i-x},v_{i-x-1}\})=c(e_{i-x-1})=\overline{c(e_x)}\;.$$
So suppose for the sake of contradiction that there is $s_0\in[n]$ such that $c(e_{s_0})\neq \overline{c(e_{i-s_0-1})}$. We show the following claim.

\begin{claim}\label{clm:reflcyclecontr}
    For every $k\in[n]$ we have that    $c(e_{s_0+2k})=c(e_{s_0})=c(e_{i-s_0-2k})$ and $c(e_{s_0+2k+1})=c(e_{s_0-1})=c(e_{i-s_0-2k-1})$.
\end{claim}

\begin{claimproof}
We show the statement by induction. First observe that it holds for $k=0$. Indeed by \eqref{eq:condneigh2} with $s=s_0$ and since $c(e_{s_0})\neq\overline{c(e_{i-s_0-1})}$ we have that $c(e_{s_0-1})=c(e_{i-s_0-1})$ and $c(e_{s_0})=c(e_{i-s_0})$. In particular this implies that $\overline{c(e_{s_0-1})}=\overline{c(e_{i-s_0-1})}\neq c(e_{s_0})$. Furthermore, by \eqref{eq:condneigh2} for $s=s_0+1$ and since $c(e_{s_0})\neq \overline{c(e_{s_0-1})}=\overline{c(e_{i-s_0-1})}$ we have $c(e_{s_0})=c(e_{i-s_0-2})$ and $c(e_{s_0+1})=c(e_{i-s_0-1})=c(e_{s_0-1})$. Now suppose the statement holds for $k\in [n]$. By \eqref{eq:condneigh2} with $s=s_0+2k+1$ and since $c(e_{s_0+2k})=c(e_{s_0})\neq\overline{c(e_{s_0-1})} =\overline{c(e_{i-s_0-2k-1})}$ we have that $c(e_{i-s_0-2k-2})=c(e_{s_0+2k})=c(e_{s_0})$. Similarly, by \eqref{eq:condneigh2} with $s=s_0+2k+2$ and since $c(e_{s_0+2k+1})=c(e_{s_0-1})\neq\overline{c(e_{s_0})} =\overline{c(e_{i-s_0-2k-2})}$ we have that $c(e_{s_0+2k+2})=c(e_{i-s_0-2k-2})=c(e_{s_0})$ and $c(e_{i-s_0-2k-3})=c(e_{s_0+2k+1})=c(e_{s_0-1})$. Finally, by \eqref{eq:condneigh2} with $s=s_0+2k+3$ and since $c(e_{s_0+2k+2})=c(e_{s_0})\neq\overline{c(e_{s_0-1})} =\overline{c(e_{i-s_0-2k-3})}$ we have that $c(e_{s_0+2k+3})=c(e_{i-s_0-2k-3})=c(e_{s_0-1})$.
\end{claimproof}

Thus $C_{n}$ is by Claim~\ref{clm:reflcyclecontr} colored with the at most two alternating directed colors $c(e_{s_0})$ and $c(e_{s_0-1})$ contradicting the assumption.
\end{proof}

\section{Proof of the Main Theorems}\label{sec:mainproof}

The following lemma shows that if copies of the graphs from Proposition~\ref{prop:rigidfamily} are glued along a cycle and the resulting graph is positive, then the resulting cyclic structure must either admit a reflection or exhibit a strong form of periodicity. This forms the core for proving our main theorems.

\begin{lem}\label{lem:auxgraphsym}
Let $k,\ell\in\mathbb N$ with $1\leq\ell\leq 2k$ and $k>1$, $\iota:[2k]\rightarrow[\ell]$, and $\tau:[2k]\rightarrow\{1,T,\sym\}$. Let $(F_1,(a_1,b_1)),\ldots,(F_\ell,(a_\ell,b_\ell))$ be edge-rooted graphs given by Proposition~\ref{prop:rigidfamily}. Let $G=G(H_1,\ldots,H_{2k})$ be the graph obtained by cyclically gluing graphs $H_i$, where each rooted graph $(H_i,(v_i,v_{i+1}))$ is isomorphic to $F_{\iota(i)}^{\tau(i)}$. Consider $W=H_1\ldots H_{2k}$ as a formal word in the letters $F_r$, $F_r^T$, and $F_r^\sym$, $r\in[\ell]$. If $G$ is positive, then one of the following holds.
\begin{enumerate}
\item A cyclic permutation of $W$ can be written as $LL^T$.\label{itm:reflection}
\item A cyclic permutation of $W$ can be written as 
\begin{enumerate}
    \item $(F_r^t)^{2k}$, for some $r\in[\ell]$, $t\in\{1,T,\sym\}$, or
    \item $(F_r^{t_1}F_s^{t_2})^{k}$, for some $r,s\in[\ell]$, $t_1,t_2\in\{1,T,\sym\}$.
\end{enumerate}\label{itm:oneortwovar}
\item The word $W$ is cyclically periodic and consists only of symmetric letters $F_r^\sym$, $r\in[\ell]$. More precisely, for every $q\in[2k]$ there exist $t\in[\ell]$, a divisor $s\mid (2k)$, and a formal word $U_q$ of even length at least $2$ formed by letters $F_r^\sym$, $r\in[\ell]$, such that, after cyclically permuting $W$ so that the sequence starts with $H_q$, the sequence has the form 
$$(H_qU_q)^s=(F_t^\sym U_q)^s=\underbrace{(F_t^\sym U_q)(F_t^\sym U_q)\cdots(F_t^\sym U_q)}_{s\text{ times}}\;.$$
\label{itm:periodic}
\end{enumerate}
\end{lem}

\begin{proof}
We call a vertex $u\in V(H_i)\setminus\{v_i,v_{i+1}\}$ an \emph{internal vertex} of $H_i$ and $\{v_i,v_{i+1}\}$ the \emph{root vertices} of $H_i$.
Since by Proposition~\ref{prop:rigidfamily}~\ref{itm:rigidf} each $F_i$ is rigid for every $i\in[\ell]$, if $F_i\cong F_j$ for some $i,j\in[\ell]$, then there exists a unique isomorphism in $\Hom(F_i,F_j)$. By a similar reasoning, if $F_i^\sym \cong F_j^\sym$, then this implies $F_i \cong F_j$, and by Proposition~\ref{prop:automsym} there are exactly two mappings in $\Hom(F_i^\sym, F_j^\sym)$ and both are isomorphisms.
For $i \in [2k]$, let $\psi_i$ denote the unique isomorphism from $H_i$ to $F_{\iota(i)}$ if $\tau(i)\neq\sym$; otherwise, let $\psi_i$ be an isomorphism from $H_i$ to $F_{\iota(i)}^\mathrm{sym}$. Note that in the latter case, the choice of isomorphism is arbitrary. For vertices $u,v\in F_{\iota(i)}^{t}$, $t\in\{1,\sym\}$, we write $u\equiv v$ if $u=v$ or $t=\sym$ and $u=\sigma(v)$, where $\sigma\in\Aut(F_{\iota(i)}^{\sym})\setminus\{id\}$. Two vertices $u \in V(H_i)$ and $v \in V(H_j)$ are called \emph{copies} if $\psi_i(u)\equiv \psi_j(v)$.

By construction, vertices in $G$ that are not copies are distinguished by their degrees or the degrees of their neighbors.

\begin{claim}\label{clm:degreepropG}
\begin{enumerate}[label=(\roman*)]
    \item $d_G(u)\neq d_G(v)$ for every internal vertex $u\in V(G)$ and every root vertex $v\in V(G)$.\label{itm:introotdeg}
    \item $d_G(u)\neq d_G(v)$ or $D_G(u)\neq D_G(v)$ for all internal vertices $u,v\in V(G)$ which are not copies.\label{itm:intneighdeg}
    \item For two root vertices $v_i,v_j\in V(G)$, if $d_G(v_i)= d_G(v_j)$ and $D_G(v_i)= D_G(v_j)$, then one of the two following cases holds.
    \label{itm:rootvrt}
\begin{enumerate}
        \item $H_{i-1}\cong H_{j-1}$ and $H_{i}\cong H_j$ and $\psi_{i-1}(v_{i-1})\equiv\psi_{j-1}(v_{j-1})$, $\psi_{i-1}(v_{i})\equiv\psi_{j-1}(v_{j})$, $\psi_{i}(v_{i})\equiv\psi_{j}(v_{j})$, and $\psi_{i}(v_{i+1})\equiv\psi_{j}(v_{j+1})$ or
        \item $H_{i-1}\cong H_{j}$ and $H_{i}\cong H_{j-1}$ and $\psi_{i-1}(v_{i-1})\equiv\psi_{j}(v_{j+1})$, $\psi_{i-1}(v_{i})\equiv\psi_{j}(v_{j})$, and $\psi_{i}(v_{i+1})\equiv\psi_{j-1}(v_{j-1})$, $\psi_{i}(v_{i})\equiv\psi_{j-1}(v_{j})$.
    \end{enumerate}
\end{enumerate}
\end{claim}

\begin{claimproof}
For~\ref{itm:introotdeg} let $u$ be an internal vertex. Then $d_G(u)=d_{F_t}(u')$ for some $t\in[\ell]$ and $u'\in V(F_t)$ by construction of $G$.
Let $v$ be a root vertex. Then by the construction of $G$ (and the definition of a symmetrization) we have $d_G(v)=d_{F_r}(v_1)+d_{F_s}(v_2)$ for some $r,s\in[\ell]$ and $v_1\in V(F_r)$, $v_2\in V(F_s)$, or $d_G(v)=d_{F_r}(v_1)+ d_{F_r}(v_2) + d_{F_s}(v_3) - 1$ for some $r,s\in[\ell]$ and $v_1, v_2\in V(F_r)$, $v_3\in V(F_s)$, or  $d_G(v)=d_{F_r}(v_1)+ d_{F_r}(v_2) + d_{F_s}(v_3) + d_{F_s}(v_4) - 2$ for some $r,s\in[\ell]$ and $v_1, v_2\in V(F_r)$, $v_3, v_4\in V(F_s)$.
Then by Proposition~\ref{prop:rigidfamily}~\ref{itm:sumdeg2} it follows that $u$ cannot have the same degree as any root vertex in $G$.

For \ref{itm:intneighdeg} let $u\in V(H_i)$ and $v\in V(H_j)$ be internal vertices which are not copies. Then $d_G(u)$ is the degree of an internal vertex in some $F_s$ and $d_G(v)$ is the degree of an internal vertex in some $F_{s'}$. If $F_s \not\cong F_{s'}$ (which also implies $H_i \not\cong H_j$), then, since $u,v$ are internal vertices, we have by Proposition~\ref{prop:rigidfamily}~\ref{itm:diffdeg} that $d_G(u) \neq d_G(v)$. So suppose that $H_i\cong H_j$ and both correspond to the same $F_s$. Let $u'$ be the copy of $u$ in $F_s$ and $v'$ the copy of $v$ in $F_s$, then $u'\neq v'$, since $u$ and $v$ are not copies. If $u'\in S_s$, then by Proposition~\ref{prop:rigidfamily}~\ref{itm:degsequneighb} we have that $d_{F_s}(u')\neq d_{F_s}(v')$. So we may assume that $u',v'\notin S_s$ and by  Proposition~\ref{prop:rigidfamily}~\ref{itm:degsequneighb} we have that $N_{F_s}(u')\cap S_s\neq N_{F_s}(v')\cap S_s$. Hence w.l.o.g. we may assume that there exists $w'\in S_s$ such that $d_{F_s}(w')\in D_{F_s}(u')$ and $d_{F_s}(w')\notin D_{F_s}(v')$. Let $w$ be the copy of $w'$ in $H_i$. Note that by construction $D_{F_s}(u')\setminus\{d_{F_s}(a_{i}),d_{F_s}(b_{i})\}\subseteq D_{G}(u)$. Since by Proposition~\ref{prop:rigidfamily}~\ref{itm:degsequneighb} the elements of $S_s$ are the vertices with the highest degrees in $F_s$, $d_{G}(w)=d_{F_s}(w')> d_{F_s}(a_i), d_{F_s}(b_i)$ and therefore $d_{G}(w)\in D_{G}(u)$. Furthermore, $D_{G}(v)\subseteq D_{F_s}(v')\cup\{d_G(v_{i}),d_G(v_{i+1})\}$ and by \ref{itm:introotdeg} we have that $d_G(w)\neq d_G(v_{i}),d_G(v_{i+1})$ and therefore $d_G(w)=d_{F_s}(w')\notin D_{G}(v)$ showing that $D_{G}(u)\neq D_G(v)$.

For~\ref{itm:rootvrt} let $v_i$ and $v_j$ be root vertices such that $d_G(v_i)=d_G(v_j)$ and $D_G(v_i)=D_G(v_j)$. Since $d_G(v_i)=d_{H_{i-1}}(v_i)+d_{H_{i}}(v_i)$ and $d_G(v_j)=d_{H_{j-1}}(v_j)+d_{H_{j}}(v_j)$, by Proposition~\ref{prop:rigidfamily}~\ref{itm:sumdeg} we have for some $s,t\in[\ell]$ and $\alpha(s), \alpha(t) \in \{1, \sym\}$ that
\begin{enumerate}
    \item $H_{i-1}\cong F_s^{\alpha(s)}\cong H_{j-1}$ and $H_{i}\cong F_t^{\alpha(t)}\cong H_{j}$ or
    \item $H_{i-1}\cong F_s^{\alpha(s)}\cong H_{j}$ and $H_{i}\cong F_t^{\alpha(t)}\cong H_{j-1}$.
\end{enumerate}  (Note that $F_s^T \cong F_s$ so here we only include $1, \sym$ for the values of $\alpha(s), \alpha(t)$).

We first consider the case $s\neq t$. Recall the definition of $S_x$ in Proposition \ref{prop:rigidfamily}. Define $D_{S_x}=\{d_{F_x}(u)\mid u\in S_x\}$, $x\in[\ell]$, and note that $x\neq y\in[\ell]$ implies that $D_{S_x}\cap D_{S_y}=\emptyset$ by Proposition~\ref{prop:rigidfamily}~\ref{itm:diffdeg} and~\ref{itm:degsequneighb}.

We also set $S=\bigcup_{x\in[\ell]}S_x$ and $D_S=\bigcup_{x\in[\ell]}D_{S_x}$. Furthermore, note that $\{d_G(v_i)\mid i\in[2k]\}\cap D_S=\emptyset$ by~\ref{itm:introotdeg}. Thus for $u\in V(G)$ and $x\in[\ell]$ we have $d_G(u)\in D_{S_x}$ if and only if $u$ is an internal vertex and a copy of a vertex from $S_x$. In the first situation where $H_{i-1}\cong F_s^{\alpha(s)}\cong H_{j-1}$ and $H_{i}\cong F_t^{\alpha(t)}\cong H_{j}$, since $s\neq t$, we have
\begin{align*}
D_G(v_i)\cap D_S&=(D_{F_{s}^{\alpha(s)}}(\psi_{i-1}(v_i))\cap D_{S_{s}})\cupdot (D_{{F_{t}^{\alpha(t)}}}(\psi_{i}(v_i))\cap D_{S_{t}})\textrm{ and}\\
D_G(v_j)\cap D_S&=(D_{F_{s}^{\alpha(s)}}(\psi_{j-1}(v_j))\cap D_{S_{s}})\cupdot (D_{F_{t}^{\alpha(t)}}(\psi_{j}(v_j))\cap D_{S_{t}})\;.
\end{align*}

Since $D_G(v_i)\cap D_S=D_G(v_j)\cap D_S$, Proposition~\ref{prop:rigidfamily}~\ref{itm:diffdeg} and~\ref{itm:degsequneighb} imply that $\{\psi_{i-1}(v_i),\psi_{i}(v_i)\}=\{\psi_{j-1}(v_j),\psi_{j}(v_j)\}$ after treating copies as the same vertex, leading to the two cases of the assertion. The second situation $H_{i-1}\cong F_s^{\alpha(s)}\cong H_{j}$ and $H_{i}\cong F_t^{\alpha(t)}\cong H_{j-1}$ is proved using the same argument.

Now, we consider the case $s=t$. If $\alpha(t)=\alpha(s)=\sym$, then $H_{i-1}\cong H_{i}\cong H_{j-1}\cong H_{j}\cong F_s^{\sym}$ and it holds for each of those graphs that their two roots are copies implying the assertion. Hence we may assume that $\alpha(s)\neq\sym$. By Proposition~\ref{prop:rigidfamily}~\ref{itm:rootdeg} we have that $d_{F_s}(a_s)<d_{F_s}(b_s)$ and therefore that $d_{F_s}(a_s)+d_{F_s}(a_s)$, $d_{F_s}(a_s)+d_{F_s}(b_s)$ and $d_{F_s}(b_s)+d_{F_s}(b_s)$ are distinct. Thus $d_G(v_i)=d_G(v_j)$ implies again $\{\psi_{i-1}(v_i),\psi_{i}(v_i)\}=\{\psi_{j-1}(v_j),\psi_{j}(v_j)\}$ leading to the two cases of the assertion.
\end{claimproof}

We call an automorphism $\varphi\in \Aut(G)$ a \emph{rotation} if there exists $t\in[2k]$ such that $\varphi(v_i)=v_{i+t}$ and $\varphi(V(H_i))=V(H_{i+t})$ for every $i\in[2k]$. Furthermore, we call $\varphi\in \Aut(G)$ a \emph{reflection} if there exists $t\in[2k]$ such that $\varphi(v_{t-i})=v_{t-1+i}$ and $\varphi(V(H_{t-i}))=V(H_{t-1+i})$ for every $i\in[2k]$.

For each $q\in[2k]$ we will analyze the union of the walk tree partition classes of the vertices of $H_q$, i.e. 
\[\mathcal{C}_q=\{u\in V(G)\mid \textrm{ there exists }v\in V(H_q)\textrm{ s.t. }R_G(u)\cong R_G(v)\}\;. \]
Additionally, we sometimes consider $\mathcal{C}_q'=\mathcal{C}_q\cup\{v_i\mid i\in[2k]\}$.
We make the following observation that $\mathcal{C}_q$ only consists of copies of $H_q$ in $G$.

\begin{claim}\label{clm:unionofH1}
Let $q\in[2k]$. There exists $J_q\subseteq [2k]$ such that $\mathcal{C}_q=\bigcup_{i\in J_q}V(H_i)$ and $H_i\cong H_q$ for every $i\in J_q$. In particular, for each $i\in J_q$ either $v_i$ is a copy of $v_q$ and $v_{i+1}$ is a copy of $v_{q+1}$ or $v_i$ is a copy of $v_{q+1}$ and $v_{i+1}$ is a copy of $v_q$.
\end{claim}

\begin{claimproof}
For simplicity we may w.l.o.g. assume that $q=1$.
Let $u\in\mathcal{C}_1$. Suppose first that $u$ is the internal vertex of some $H_i$, $i\in[2k]$. Then $R_G(v)\cong R_G(u)$ for some $v\in V(H_1)$ and by Proposition~\ref{prop:isoimplsamendeg} and Claim~\ref{clm:degreepropG}~\ref{itm:introotdeg} and~\ref{itm:intneighdeg} we have that $v$ and $u$ are copies and in particular $H_1\cong H_i$. By Proposition~\ref{prop:isoimplsameneighiso} there must be a bijection between $\{R_G(z)\mid z\in N_G(v)\}$ and $\{R_G(z)\mid z\in N_G(u)\}$ such that the corresponding pairs are isomorphic as rooted trees. By Proposition~\ref{prop:isoimplsamendeg} and Claim~\ref{clm:degreepropG}~\ref{itm:introotdeg} and~\ref{itm:intneighdeg} this bijection must map internal vertices of $H_1$ to their copy in $H_i$ and therefore the internal vertices of $H_i$ which are in $N_G(u)$ are also in $\mathcal{C}_1$. Since by Proposition~\ref{prop:rigidfamily}~\ref{itm:connectivity} $H_i-\{v_i,v_{i+1}\}$ is connected and by repeating this argument, this implies that all internal vertices of $H_i$ are contained in $\mathcal{C}_1$. By Proposition~\ref{prop:rigidfamily}~\ref{itm:notcommonneigh} there is an internal vertex $u'\in V(H_i)$ which is adjacent to $v_i$ but not to $v_{i+1}$. By the previous observation we have that $R_G(v')\cong R_G(u')$, where $v'$ is the copy of $u'$ in $H_1$. Therefore $v'$ is adjacent to the copy of $v_i$ in $H_1$ which we denote by $v_x$, where $x\in\{1,2\}$. Again by Proposition~\ref{prop:isoimplsameneighiso} there must be a bijection between $\{R_G(z)\mid z\in N_G(v')\}$ and $\{R_G(z)\mid z\in N_G(u')\}$ such that the corresponding pairs are isomorphic as rooted trees. By Proposition~\ref{prop:isoimplsamendeg} and Claim~\ref{clm:degreepropG}~\ref{itm:introotdeg} and~\ref{itm:intneighdeg} this bijection must map $v_i$ to $v_x$ (since those are the only root vertices in their respective neighbourhoods) implying that $v_i\in[v_x]\subseteq\mathcal{C}_1$. Repeating this argument for a neighbour of $v_{i+1}$ which is not a neighbour of $v_i$ leads to $v_{i+1}\in[v_y]\subseteq \mathcal{C}_1$, where $y\in\{1,2\}\setminus\{x\}$ and that $v_{i+1}$ is a copy of $v_y$.

Now suppose that $v_i\in\mathcal{C}_1$ is a root vertex. By Proposition~\ref{prop:isoimplsamendeg} and Claim~\ref{clm:degreepropG}~\ref{itm:introotdeg} $v_i$ must be in the walk-tree partition class of a root vertex and therefore $v_i\in [v_x]$ for some $x\in\{1,2\}$. Now consider an internal vertex $u\in N_{H_1}(v_x)$. By Proposition~\ref{prop:isoimplsameneighiso} there must be a bijection between $\{R_G(z)\mid z\in N_G(v_i)\}$ and $\{R_G(z)\mid z\in N_G(v_x)\}$ such that the corresponding pairs are isomorphic as rooted trees. By Proposition~\ref{prop:isoimplsamendeg} and Claim~\ref{clm:degreepropG}~\ref{itm:introotdeg} $u$ must be mapped to an internal vertex which is a copy of $u$. This means there is an internal vertex $u'\in V(H_{i-1})$ such that $R(u)=R(u')$ and in particular $H_{i-1}\cong H_1$ or there is an internal vertex $u'\in V(H_{i})$ such that $R(u)=R(u')$ and in particular $H_{i}\cong H_1$. In any case repeating the above argument we get that $V(H_{i-1})\subseteq\mathcal{C}_1$ or $V(H_{i})\subseteq\mathcal{C}_1$ and the claim follows.
\end{claimproof}

We also need a claim which helps us to analyze the structure of $\mathcal{C}_q$ if all copies of $H_q$ in it are oriented the same way.

\begin{claim}\label{lem:rotation}
Let $q\in[2k]$.
\begin{enumerate}[label=(\roman*)]
\item Suppose that $G$ is not composed of alternating graphs $H_1 = F_a^\sym$ and $H_2 = F_b^\sym$. If $j\in J_q$ with $v_{j}\in[v_q]$ and $v_{j+1}\in[v_{q+1}]$, then $\varphi:v_{x}\mapsto v_{j-q+x}$ induces a rotation of $G$.\label{itm:onerot}
\item Suppose that for every $i\in J_q$ we have $v_{i}\in[v_q]$ and $v_{i+1}\in[v_{q+1}]$.
 Then  $G[\mathcal{C}_q']$, where $\mathcal{C}_q'=\mathcal{C}_q\cup\{v_i \mid i\in[2k]\}$ is a sun graph. In other words, $G[\mathcal{C}_q']\cong SG(F_{\iota(q)}^{\tau(q)},N_q,m_q)$ for some $m_q\in\mathbb N_0$ with $2k-1\geq m_q\geq 0$ and $(1+m_q)\mid 2k$ and $N_q=2k/(m_q+1)$.\label{itm:manyrotsun}
 \end{enumerate}
\end{claim}

\begin{claimproof} 
For simplicity we assume w.l.o.g. that $q=1$ and set $J=J_1$ and $\mathcal{C}=\mathcal{C}_q$.
We first make the following three observations.

\begin{subclaim}\label{claim:int}
Let $i,j\in[2k]$.
If $R_G(v_i)\cong R_G(v_j)$ and $R_G(v_{i+1})\cong R_G(v_{j+1})$, then  $R_G(v_{i+s})\cong R_G(v_{j+s})$ for every $s\in[2k]$.
\end{subclaim}

\begin{subclaimproof}
Since $R_G(v_{i+1})\cong R_G(v_{j+1})$, by Proposition~\ref{prop:isoimplsameneighiso} there must be a bijection between $\{R_G(z)\mid z\in N_G(v_{i+1})\}$ and $\{R_G(z)\mid z\in N_G(v_{j+1})\}$ such that the corresponding pairs are isomorphic as rooted trees. Since $v_{i}$ and $v_{i+2}$ are the only root vertices in $N_G(v_{i+1})$ and $v_{j}$ and $v_{j+2}$ the only root vertices in $N_G(v_{j+1})$ we have by Proposition~\ref{prop:isoimplsamendeg} and Claim~\ref{clm:degreepropG}~\ref{itm:introotdeg} that $R_G(v_{i})\cong R_G(v_{j})$ and $R_G(v_{i+2})\cong R_G(v_{j+2})$ or $R_G(v_{i+2})\cong R_G(v_{j})$ and $R_G(v_{j+2})\cong R_G(v_{i})$.
Note that in the latter case we also get together with the assumption $R_G(v_{i})\cong R_G(v_{j})$ that $R_G(v_{i+2})\cong R_G(v_{j+2})$. Repeating this argument inductively shows that $R_G(v_{i+s})\cong R_G(v_{j+s})$ for every $s\in[2k]$.
\end{subclaimproof}

\begin{subclaim}
    \label{clm:wtpc_rot}
Suppose that $i,j\in[2k]$ with $i\neq j$ and such that  $R_G(v_{i+s})\cong R_G(v_{j+s})$ for every $s\in[2k]$. Then either $\varphi:v_x\mapsto v_{j-i+x}$ induces a rotation of $G$ mapping $H_i$ to $H_j$, or $G$ is alternating between two graphs $H_1 = F_a^\sym$ and $H_2 = F_b^\sym$. In either case, $G$ has a non-trivial rotation.
\end{subclaim}

\begin{subclaimproof}
We want to deduce that $H_{i+s}\cong H_{j+s}$ for every $s\in[2k]$. For this we define an auxiliary oriented cycle $\widetilde{C}$ with vertex set $V(\widetilde{C})=(v_1,\cdots,v_{2k})$. The ``orientation" of each edge $e_i=\{v_i,v_{i+1}\}$, $i\in[2k]$, is given by $\tau(i)$ where $``1", ``T", ``\sym"$ corresponds to $1, -1, 0 \in \mathbb{F}_3$ respectively. Furthermore we color each edge $e_i$ with $F_{\iota(i)}$. Now by Proposition~\ref{prop:isoimplsamendeg} and Claim~\ref{clm:degreepropG}~\ref{itm:rootvrt} the conditions of Proposition~\ref{prop:cyclerotation} are fulfilled and one of the following two situations will happen. (1) $\varphi:v_{s}\mapsto v_{s+j-i}$ is a color preserving automorphism of $\widetilde{C}$ and by definition of $\widetilde{C}$ then also induces a rotation of $G$; (2) $G$ is alternating between two symmetrized graphs (not necessarily distinct).
\end{subclaimproof}

\begin{subclaim}\label{clm:rot_wtpc}
If there exists a rotation $\varphi\in \Aut(G)$ mapping $H_i$ to $H_j$, then for every pair of copies $x\in V(H_i)$ and $y\in V(H_j)$ we have $R_G(x)\cong R_G(y)$.
\end{subclaim}

\begin{subclaimproof}
Given $\varphi$, note that $\varphi|_{V(H_i)}$ is an isomorphism mapping vertices from $H_i$ to their copy in $H_j$. We define $\psi:V(R_G(x))\rightarrow V(R_G(y))$ by $\psi((w_1\cdots w_r))\mapsto (\varphi(w_1)\cdots\varphi(w_r))$ for every walk $(w_1\cdots w_r)$ starting at $x$. Since $\varphi(x)=y$ and $\varphi$ is a homomorphism, it follows that $\psi$ is well-defined, i.e. $(\varphi(w_1)\cdots\varphi(w_r))$ is a walk starting at $y$. Furthermore, from the fact that $\varphi$ is an isomorphism it follows that $(w_1\cdots w_r)$ is a walk extending $(w_1\cdots w_{r-1})$ if and only if $(\varphi(w_1)\cdots \varphi(w_r))$ is a walk extending $(\varphi(w_1)\cdots \varphi(w_{r-1}))$, and therefore $\psi$ is also an isomorphism of the rooted trees $R_G(x)$ and $R_G(y)$.
\end{subclaimproof}

We can now turn to proving Claim~\ref{lem:rotation}.
Item~\ref{itm:onerot} follows immediately from combining Observations~\ref{claim:int} and~\ref{clm:wtpc_rot}.

For~\ref{itm:manyrotsun} we consider the elements of $J$ modulo $2k$ and w.l.o.g. we may assume by relabeling the vertices that $0\in J$. If $J=\{0\}$, then $G[\mathcal C']\cong SG(F_{\iota(1)}^{\tau(1)},1,2k-1)$, and there is nothing to prove. Thus we may assume that $|J|\geq2$.
Choose $d=\min J\setminus\{0\}$ and note that by the assumption on $J$ for~\ref{itm:manyrotsun}, Observation~\ref{claim:int} and Observation~\ref{clm:wtpc_rot} there is a rotation of $G$ by $d$. This implies by Observation~\ref{clm:rot_wtpc} that also $2d\in J$ and iterating this argument gives that $sd$ mod $2k$ is in $J$ for every $s\in\mathbb N$. Furthermore, we have that $d\mid a$ for every $a\in J$, as otherwise $a=qd+r$ for some $1\leq r\leq d-1$ and then by the assumption and Observation~\ref{clm:wtpc_rot} we also have $r\equiv a+(2k-q)d$ in $J$ contradicting the minimality of $d$. Thus overall we have that $J=(d\mathbb Z)/(2k\mathbb Z)$. Therefore $G[\mathcal{C}']$ is a sun graph with equally spread out attached copies of $F_{\iota(1)}^{\tau(1)}$ as desired, i.e. $G[\mathcal{C}']\cong SG(F_{\iota(1)}^{\tau(1)},2k/(m+1),m)$ for some $m\geq 0$.
\end{claimproof}

And we need a last claim for the case that $\mathcal{C}_q$ contains two copies of $H_q$ with different orientation. Then we will find a reflection of $G$.

\begin{claim}\label{lem:case2}
    Let $q\in[2k]$.
    Suppose that there exists $i\in J_q$ such that $v_i\in[v_{q+1}]$ and $v_{i+1}\in[v_q]$. Then $W = (F_r^t)^{2k}$ for some $r\in[\ell]$ and $t\in\{1,T,\sym\}$ or $\varphi_{i}:v_x\mapsto v_{q+i+1-x}$ induces a reflection of $G$.
\end{claim}
\begin{claimproof}
For simplicity we assume w.l.o.g. that $q=1$.
We first make the following observation.
\begin{subclaim}\label{clm:oppositeinduc}
$R_G(v_{1+s})\cong R_G(v_{i+1-s})$ for all $s\in [2k]$.
\end{subclaim}

\begin{subclaimproof}
 Since $R_G(v_{i})\cong R_G(v_{2})$, by Proposition~\ref{prop:isoimplsameneighiso} there must be a bijection between $\{R_G(z)\mid z\in N_G(v_{i})\}$ and $\{R_G(z)\mid z\in N_G(v_{2})\}$ such that the corresponding pairs are isomorphic as rooted trees. Denote by $\{v_{i-1},v_{i+1},x_1\cdots,x_{d_1}\}$ the neighbours of $v_{i}$ and by $\{v_{1},v_{3},y_1\cdots,y_{d_2}\}$ the neighbours of $v_{2}$. Since by Claim~\ref{clm:degreepropG}~\ref{itm:introotdeg} $d_G(x_s)\neq d_G(v_{1}),d_G(v_{3})$ for every $s\in[d_1]$ and $d_G(y_s)\neq d_G(v_{i-1}),d_G(v_{i+1})$ for every $s\in[d_2]$, we have by Proposition~\ref{prop:isoimplsamendeg} that
$$R_G(v_{i-1})\cong R_G(v_{1})\textrm{ and }R_G(v_{i+1})\cong R_G(v_{3})\textrm{ or}$$
$$R_G(v_{i+1})\cong R_G(v_{1})\textrm{ and }R_G(v_{i-1})\cong R_G(v_{3})\;.$$ In the latter case we immediately have $R_G(v_{3})\cong R_G(v_{i-1})$. In the former case we also get together with $R_G(v_{1})\cong R_G(v_{i+1})$ that $R_G(v_{3})\cong R_G(v_{i-1})$. Repeating this argument inductively shows that $R_G(v_{1+s})\cong R_G(v_{i+1-s})$ for every $s\in[2k]$.
\end{subclaimproof}

Now, we define an auxiliary oriented cycle $\widetilde{C}$ with vertex set $(v_1,\cdots,v_{2k})$ and $c(\{v_i,v_{i+1}\})=(\tau(i),F_{\iota(i)})$ where we slightly abuse notation and let the values of $\tau$: $``1", ``T", ``\sym"$ correspond to $1, -1, 0 \in \mathbb{F}_3$, respectively.
If there exist $c'$ and $c''$ such that $c(e_{2i+1})=c'$ and $c(e_{2i})=c''$ for every $i\in[k]$ then we are in Outcome~\ref{itm:oneortwovar} of the lemma.

Hence we may assume that there are no such $c'$ and $c''$ and also by Claim~\ref{clm:oppositeinduc}, Proposition~\ref{prop:isoimplsamendeg} and Claim~\ref{clm:degreepropG}~\ref{itm:rootvrt} the conditions of Proposition~\ref{prop:cyclereflection} are fulfilled and $\varphi:v_{x}\mapsto v_{1+i+1-x}$ is a color preserving automorphism of $\widetilde{C}$ and by definition of $\widetilde{C}$ also induces a reflection of $G$.
\end{claimproof}

We can now perform the main analysis of $G$ which we split into two cases. First, we will assume that for every $q\in[2k]$ all the copies of $H_q$ in $\mathcal{C}_q$ are oriented the same way which will lead to rotational symmetries which we can exploit for a contradiction. Then we will consider the case that some class $\mathcal{C}_q$ contains copies of opposite orientation which will lead us to a reflection of $G$ showing that $W$ is a palindrome after cyclic permutation. From now on, for every $q\in[2k]$ we fix the set $J_q\subseteq[2k]$ given by Claim~\ref{clm:unionofH1}.\\

\noindent{\bf Case 1:} For every $q\in[2k]$ and every $i\in J_q$ we have that $v_i\in[v_q]$ and $v_{i+1}\in[v_{q+1}]$. We apply Claim~\ref{lem:rotation} to $q$ and get $G[\mathcal{C}_q']\cong SG(F_{\iota(q)}^{\tau(q)},N_q,m_q)$ for some $2k-1\geq m_q\geq 0$ and $N_q=2k/(m_q+1)$. First, suppose that $m_q=0$. Then $W=(F_r^t)^{2k}$ for some $r\in[\ell]$ and $t\in\{1,T,\sym\}$ corresponding to Outcome~\ref{itm:oneortwovar}.

Now, suppose that $m_q=1$. Then there exist $r\in[\ell]$ and $t_r\in\{1,T,\sym\}$ such that
for every $i\in[2k]$ we have $(H_{q+2i},(v_{q+2i},v_{q+2i+1}))\cong (F_r,(a_r,b_r))^{t_r}$. Furthermore, we are in the case that $R_G(v_{q+2i})\cong R_G(v_{q})$ and $R_G(v_{q+2i+1})\cong R_G(v_{q+1})$ for every $i\in[2k]$ which implies by Proposition~\ref{prop:isoimplsamendeg} that $d_G(v_{q+2i+1})=d_G(v_{q+1})$ for every $i\in[k]$. Therefore by Proposition~\ref{prop:rigidfamily}~\ref{itm:sumdeg} there is $s\in[\ell]$ and $t_s\in\{1,T,\sym\}$ such that $(H_{q+2i+1},(v_{q+2i+1},v_{q+2i+2}))\cong (F_s,(a_s,b_s))^{t_s}$ for every $i\in[k]$ and such that those are all oriented the same way. Hence $W\eqshift(F_r^{t_1}F_s^{t_2})^{k}$ for some $t_1,t_2\in\{1,T,\sym\}$; note that $r=s$ is possible. Again this corresponds to Outcome~\ref{itm:oneortwovar}.

Finally, we assume that $m_q\geq 2$ and recall that we assume $k>1$. Since $G$ is positive by assumption and, since $\mathcal{C}_q'=\bigcup_{i\in[2k]}[v_i]\cup\bigcup_{x\in V(H_q)}[x]$ is a union of walk-tree partition classes of $G$, we have by Lemma~\ref{lem:indwalktrcl} that $G[\mathcal{C}_q']$ is also positive.
However, by Lemma~\ref{lem:sunnotpos}~\ref{itm:lodd}, if $m_q$ is odd, then $G[\mathcal{C}_q']$ is not positive. Thus $m_q$ is even. If $H_q$ is rigid, then by Lemma~\ref{lem:sunnotpos}~\ref{itm:levenrigid} $G[\mathcal{C}_q']$ is not positive. Hence $H_q$ is a symmetrized graph by the fact that each $F_i$ is rigid.
Note for the application of Lemma~\ref{lem:sunnotpos} that Proposition~\ref{prop:rigidfamily}~\ref{itm:rigidf} and~\ref{itm:connectivity} and Proposition~\ref{prop:automsym} ensure the required conditions for $H_q$.
Therefore we are only left with the following scenario:
Each letter of $W$ is a symmetrized graph, and for each $q\in[2k]$ we have $G[\mathcal{C}_q']$ is a sun graph $SG(H_q,N_q,m_q)$ for $m_q$ even and hence also $N_q=2k/(m_q+1)$ even. By Claim~\ref{lem:rotation}~\ref{itm:onerot} this corresponds to Outcome~\ref{itm:periodic}.\\

\noindent{\bf Case 2:} There exists $q\in[2k]$ such that there exists $i\in J_q$ such that $v_i\in[v_{q+1}]$ and $v_{i+1}\in[v_q]$.
We apply Claim~\ref{lem:case2} and have $W=(F_r^t)^{2k}$ or get a reflection $x\mapsto q+i+1-x$ for $G$. The former corresponds to Outcome~\ref{itm:oneortwovar} and for the latter consider the following. If $q+i+1$ is even, then this implies that after a cyclic permutation $W=LL^T$ for some word $L$ corresponding to Outcome~\ref{itm:reflection}. So we may suppose that every reflection $\varphi_p:v_x\mapsto v_{p-x}$ of $G$ has $p$ odd and in particular that $q+i+1$ is odd. W.l.o.g. we may assume after relabeling that $q+i=2k$ and the reflection for $G$ is induced by $\varphi_0:v_x\mapsto v_{1-x}$ with indices taken mod $2k$. Then this implies in particular that $H_0=H_{2k}$ and $H_k$ are symmetrized graphs. Furthermore, since $\varphi_0(v_0)=v_1$ and automorphic images lie in the same walk tree partition class, we get $[v_0]=[v_1]$. In particular this implies that for every $i\in J_0$ we have $v_{i}\in[v_0]$ and $v_{i+1}\in[v_{1}]$. We may therefore apply Claim~\ref{lem:rotation}~\ref{itm:manyrotsun} and get that $G[\mathcal{C}_0']\cong SG(F_{\iota(0)}^{\sym},2k/(m+1),m)$ for some $m\in\mathbb N$ with $2k-1\geq m\geq 0$ and $(1+m)\mid 2k$. 

The case $m=0$ corresponds to Outcome~\ref{itm:oneortwovar}. If $m=1$, then we have $[v_i]=[v_j]$ for all $i,j\in[2k]$ and therefore by Proposition~\ref{prop:isoimplsamendeg} and Claim~\ref{clm:degreepropG}~\ref{itm:rootvrt} that $W\eqshift (F_r^\sym F_s^\sym)^{k}$ for some $r,s\in[\ell]$, again corresponding to Outcome~\ref{itm:oneortwovar}.
If $m\geq 2$ is odd, then by Lemma~\ref{lem:sunnotpos}~\ref{itm:lodd}, we have that $G[\mathcal{C}_0']$ is not positive contradicting that a union of walk-tree partition classes of $G$ is positive by Lemma~\ref{lem:indwalktrcl}. Thus we may assume that $k\geq m\geq 2$ is even. By Claim~\ref{lem:rotation}~\ref{itm:onerot} for every copy of $H_0$ in $G[\mathcal{C}_0']$ there is a rotation of $G$ mapping $H_0$ to this copy. Therefore cyclically permuting $W$ such that the word begins with $H_0$ results in
\begin{equation}\label{eq:xqxqt}
(F_{\iota(0)}^{\sym} Q)^{2k/(m+1)}
\end{equation}
for some word $Q$ of length $m$. The reflection $\varphi_0$ of $G$ with reflection axis through $H_0$ then implies that $Q=Q^T$ and, since $m$ is even, $Q=PP^T$ for some word $P$.
Hence $W\eqshift(F_{\iota(0)}^{\sym} PP^T)^{2k/(m+1)}\eqshift(P^TF_{\iota(0)}^{\sym} P)^{2k/(m+1)}$.
Since $m$ is even, $m+1$ is odd. As $(m+1)\mid 2k$, it follows that $(m+1)\mid k$, and therefore $2k/(m+1)$ is even. Thus 
$$W\eqshift (P^TF_{\iota(0)}^{\sym}P)^{k/(m+1)}((P^TF_{\iota(0)}^{\sym} P)^{k/(m+1)})^T=LL^T$$ corresponding to Outcome~\ref{itm:reflection}.
\end{proof}

We can now apply this lemma to show Theorem~\ref{thm:realsym}.

\begin{proof}[Proof of Theorem~\ref{thm:realsym}]
First suppose that $P$ is symmetric and there exists $j\in[k]$ such that $P'=\prod_{i=j+1}^{j+k}X_{\iota(i)}^{\tau(i)}=LL^T$ for some symbolic matrix product $L$. Let $A_1,\cdots,A_{\ell}\in\RR^{n\times n}$. Then assigning $A_i$ to $X_i$, for every $i\in[\ell]$, leads to $P'(A_1,\cdots,A_{\ell})=L(A_1,\cdots,A_{\ell})L(A_1,\cdots,A_{\ell})^T$ which has  only nonnegative eigenvalues. Set $A\coloneqq\prod_{i=j+1}^k A_{\iota(i)}^{\tau(i)}$ and $B\coloneqq\prod_{i=1}^j A_{\iota(i)}^{\tau(i)}$. Then we get by using Fact~\ref{Fact:ABBA} that $P(A_1,\cdots,A_{\ell})=BA$ and $P'(A_1,\cdots,A_{\ell})=AB$ have the same eigenvalues and therefore all eigenvalues of $P(A_1,\cdots,A_{\ell})$ are real and nonnegative.
The other case follows similarly. 
If  there exists $j \in [k]$ such that  $\prod_{i=j+1}^{j+k}X_{\iota(i)}^{\tau(i)}$ can be written as $L' L'^T X_{\iota(j+k)}^\sym$, where $L'= \prod_{i=j+1}^{j+\tfrac{k-1}{2}}X_{\iota(i)}^{\tau(i)}$, then, since $L' L'^T$ is symmetric and positive semidefinite for any assignment of matrices and  $X_{\iota(j+k)}^\sym$ is a symmetric matrix, it follows from Fact~\ref{Fact:AAt} and Fact~\ref{Fact:ABBA} that the product $P$ is real-eigenvalued.\\

Now we consider the other more difficult direction. Suppose that $P = X_{\iota(1)}^{\tau(1)}\cdots X_{\iota(k)}^{\tau(k)}$ is real-eigenvalued with $\ell$ variables and of degree $k$.
Note that if the degree of $P$ is one, then $P = X^\sym$ (corresponding to Item~\ref{itm:realsymsymplusvar} of the theorem), as otherwise $P=X$ or $P=X^T$ and assigning any matrix from $\mathbb R^{n\times n}$ which does not have real eigenvalues contradicts that $P$ is real-eigenvalued. Hence we may from now on assume that the degree $k$ of $P$ is at least two.
We choose a family of edge-rooted graphs $\{(F_1,(a_1,b_1)),\cdots,(F_{\ell},(a_\ell,b_\ell))\}$ given by Proposition~\ref{prop:rigidfamily}.
Now let $H_1,\dots,H_k$ be a family of disjoint graphs such that $H_i\cong F_{\iota(i)}^{\tau(i)}$ for every $i\in[k]$. Consider the graph $G=G^2(H_1,\cdots,H_k)$ (defined in~\eqref{eq:Gsquare}) and the subgraph $C=G[v_1,\cdots,v_{2k}]$ which is isomorphic to the cycle of length $2k$. By Proposition~\ref{pro:cyclegluepositive}~\ref{itm:realeg} we have that $G$ is positive. Therefore we can apply Lemma~\ref{lem:auxgraphsym} to $G$. 

If Item \ref{itm:reflection} of the lemma applies, a cyclic permutation of $P^2$ can be written as $QQ^T$. Therefore $P^2$ has a reflection axis between $X_{\iota(i)}^{\tau(i)}$ and $X_{\iota(i+1)}^{\tau(i+1)}$ for some $i\in[2k]$ and indices modulo $k$. This geometrically projects to a reflection on $P$. Thus $P \eqshift L L^T$, if $k$ is even, or $P \eqshift L L^T X_a^\sym$, if $k$ is odd. Those are exactly the two cases of the theorem, and we are done.

So assume that Item~\ref{itm:oneortwovar} of Lemma~\ref{lem:auxgraphsym} applies. Suppose we have that $P=(X_1^t)^k$ for some $t\in\{1,T,\sym\}$. If $X_1$ is not necessarily symmetric, then $P=X_1^k$ or $P=(X_1^T)^k$. In this case choosing $M$ as in \eqref{eq:rotationm}, but with $\theta=\pi/(2k)$, shows that $P(M)$ has the eigenvalue $i$ and hence $P$ is not real-eigenvalued. Therefore $P = (X_1^\sym)^k$ corresponding to Item~\ref{itm:realsymsym} of the theorem, if $k$ is even, and Item~\ref{itm:realsymsymplusvar}, if $k$ is odd. The second case is that $P=(X_1X_1^T)^{k/2}$ or $P=(X_1^TX_1)^{k/2}$ which are both symmetric. Lastly, if $P=(X_1^{t_1}X_2^{t_2})^{k/2}$ for some $t_1,t_2\in\{1,T,\sym\}$, then we choose the two symmetric matrices
\begin{equation}\label{eq:l1notrealev}
A\coloneqq\begin{bmatrix}
1 & 0\\
0 & -1
\end{bmatrix}\;\textrm{ and }\;B\coloneqq\begin{bmatrix}
\cos(\theta) & \sin(\theta)\\
\sin(\theta) & -\cos(\theta)
\end{bmatrix}
\end{equation}
where $\theta=\pi/k$. This shows that $P(A,B)$ has the eigenvalue $i$ and hence $P$ is not real-eigenvalued.

Thus the leftover case is that Item~\ref{itm:periodic} of Lemma~\ref{lem:auxgraphsym} applies and in particular that the product consists only of symmetric variables. To address this scenario, note first that by Fact~\ref{Fact:AAt} our given product $P(X_1,\dots,X_\ell)=\prod_{i\in[k]}X_{\iota(i)}^\sym$ has only real eigenvalues for any assignment of symmetric positive semidefinite matrices if and only if $ P'(Y_1, \dots, Y_\ell) = \prod_{i\in[k]}Y_{\iota(i)}Y_{\iota(i)}^T$ has only real eigenvalues for any assignment of real matrices. Therefore, since $P$ is real-eigenvalued, it follows that $P'$ is also real-eigenvalued.

Furthermore, note that what we did so far already proves the equivalence of Item~\ref{itm:realrealsym} and Item~\ref{itm:realrealrealev} of Theorem~\ref{thm:realreal}, since the only leftover case considers products which are restricted to symmetric variables.
Since the variables $Y_i$ are not necessarily symmetric, we therefore know that $P'$ is real-eigenvalued if and only if $P'$ is symmetric in terms of the variables $Y_1, \dots, Y_\ell$. 

Combining the observation that $P'$ is real-eigenvalued with the proven implication of Theorem~\ref{thm:realreal} we get that $P' \eqshift L L^T$ where $L$ is a formal product in terms of variables $Y_i$ and $Y_i^T$. If $L$ can again be written as a formal product of $X_i$'s by formally substituting $Y_i^TY_i = Y_i Y_i^T = X_i^\sym$ for every $i\in[\ell]$, then also $P$ is symmetric, as desired.
If $L$ cannot be written as a formal product of $X_i^\sym$'s, then $L$ can be written as one of the following two possibilities:
    \begin{enumerate}
        \item $L = Y_r^{\tau_1} L'(X_1^\sym, \dots, X_\ell^\sym) Y_s^{\tau_2}$ where $L'(X_1^\sym, \dots, X_\ell^\sym)$ is a formal product of variables $X_i^\sym$ and $r,s\in[\ell]$, $\tau_1, \tau_2 \in \{1, T\}$; or
        \item $L = Y_r^{\tau_1} L'(X_1^\sym, \dots, X_\ell^\sym)$ where $L'(X_1^\sym, \dots, X_\ell^\sym)$ is a formal product of variables $X_i^\sym$ and $r\in[\ell]$, $\tau_1\in \{1, T\}$.
    \end{enumerate}
    In the latter case, we have that
    \begin{align*}
    P \eqshift P' &\eqshift Y_r^{\tau_1} L'(X_1^\sym, \dots, X_\ell^\sym) (Y_r^{\tau_1} L'(X_1^\sym, \dots, X_\ell^\sym))^T\\
        &\eqshift L'(X_1^\sym, \dots, X_\ell^\sym) L'(X_1^\sym, \dots, X_\ell^\sym)^T (Y_r^{\tau_1})^T Y_r^{\tau_1} \\
        &= L'(X_1^\sym, \dots, X_\ell^\sym) L'(X_1^\sym, \dots, X_\ell^\sym)^T X_r^\sym\;,
    \end{align*}
    which corresponds to Case~\ref{itm:realsymsymplusvar} of the assertion in Theorem \ref{thm:realsym}.

   So from now on we may assume that for every $L$ with $P' \eqshift L L^T$, we have that $L = Y_r^{\tau_1} L'(X_1^\sym, \dots, X_\ell^\sym) Y_s^{\tau_2}$ for some $r,s \in [\ell]$ and $\tau_1, \tau_2 \in \{1, T\}$, and such that $L'(X_1^\sym, \dots, X_\ell^\sym)$ is a formal product of variables $X_i$. Without loss of generality, we may relabel the variables and assume $r = 1$ and get
\begin{equation}
    P \eqshift X_1^\sym L'(X_1^\sym, \dots, X_\ell^\sym) X_{\iota(k/2+1)}^\sym L'(X_1^\sym, \dots, X_\ell^\sym)^T. \label{eq:decomp}
\end{equation}

Consider $H_1$ in $G$ corresponding to the first $X_1^\sym$ in (\ref{eq:decomp}). By the strong cyclic periodicity of the word $H_{\iota(1)}^{\tau(1)}\ldots H_{\iota(k)}^{\tau(k)}$ we therefore have, starting with the same $X_1^\sym$ as in (\ref{eq:decomp}), that we can write
\begin{equation} P^2 \eqshift (X_1^\sym U)^s\label{eq:sunlike} \end{equation}
 for some positive integer $s$ dividing $k$ and $U$ being a formal product of even length $2k/s-1$.
 Now we compare the formal products (\ref{eq:decomp}) and (\ref{eq:sunlike}). Since $P^2$ starts with $X_1^\sym U$, this means that $U$ (as an ordered sequence) is the prefix of the sequence of variables $X_i^\sym$'s appearing in $L' X_{\iota(k/2+1)}^\sym L'^T X_1^\sym L' X_{\iota(k/2+1)}^\sym L'^T$. Similarly, since $P^2$ ends with $U$ by (\ref{eq:sunlike}), and $P$ has the form (\ref{eq:decomp}), this means $U$ is also a suffix of $L' X_{\iota(k/2+1)}^\sym L'^T X_1^\sym L' X_{\iota(k/2+1)}^\sym L'^T$. However, since the sequence of variables $X_i$ appearing in $L' X_{\iota(k/2+1)}^\sym L'^T X_1^\sym L' X_{\iota(k/2+1)}^\sym L'^T$ is a palindrome (i.e., the $i$-th letter is the same as the $i$-th to last letter for all $i$), it follows that for some formal product $W$ of variables $X_i^\sym$ we have that $U$ can be written as either $W W^T$, or $W X_j^\sym W^T$ for some $j\in[\ell]$. Since the length of $U$ is even, only the former case applies and $U$ can be written as $W W^T$.
 Therefore $P^2 \eqshift (X_1^\sym W W^T)^s$. Since  $P^2$ has even length and $(X_1^\sym W W^T)$ has odd length, this implies that $s$ is even. Therefore $P \eqshift (X_1^\sym W W^T)^{s/2}$ and
\[
P \eqshift \begin{cases}
\left( W^T X_1^\sym W \right)^{s/4} \left[ \left( W^T X_1^\sym W \right)^{s/4} \right]^T & \text{if } s/2 \text{ is even}, \\
X_1^\sym \left[ W \left( W^T X_1^\sym W \right)^{(s/2-1)/2} \right] \left[ W \left( W^T X_1^\sym W \right)^{(s/2-1)/2} \right]^T & \text{if } s/2 \text{ is odd}.
\end{cases}
\]
In the first case, $P$ is symmetric, and the second case corresponds to Item~\ref{itm:realsymsymplusvar} of the theorem.
We thus completed the proof of Theorem \ref{thm:realsym}.
\end{proof}

It is now easy to derive Corollary~\ref{cor:sympossemidef} from Theorem~\ref{thm:realsym}.

\begin{proof}[Proof of Corollary~\ref{cor:sympossemidef}]
If $P$ is symmetric, then a cyclic permutation of $P$ evaluates by Fact~\ref{Fact:AAt} to a symmetric positive semidefinite matrix for all valid substitutions. Hence by Fact~\ref{Fact:ABBA} also $P$ is spectrally nonnegative. Conversely, suppose that $P$ is spectrally nonnegative. It is therefore real-eigenvalued, allowing us to apply Theorem~\ref{thm:realsym}. Suppose for contradiction that $P$ falls into Case~\ref{itm:realsymsymplusvar}. Assigning the $1 \times 1$ scalar matrix $X_i=(1)$ for every $i \neq j$ and $X_{j}=(-1)$ yields the evaluation $W=(-1)$. This matrix has a strictly negative eigenvalue, contradicting the assumption that $P$ is spectrally nonnegative. Thus, Case~\ref{itm:realsymsymplusvar} is impossible, leaving Case~\ref{itm:realsymsym}, meaning $P$ must be symmetric.
\end{proof}

The proof of Theorem~\ref{thm:realsym2} follows the same strategy as the proof of Theorem~\ref{thm:realsym}. We will mostly highlight the differences.

\begin{proof}[Proof of Theorem~\ref{thm:realsym2}]
First suppose that $P$ is symmetric. Then assigning real matrices $A_i\in\RR^{n\times n}$, $i\in[\ell]$, and evaluating the product leads to a matrix with only nonnegative real eigenvalues by Fact~\ref{Fact:AAt} and Fact~\ref{Fact:ABBA} and therefore a nonnegative trace. Thus $P$ is trace nonnegative.

Now we consider the other more difficult direction. Suppose that $P$ is trace nonnegative. Note that $P$ has even degree, as otherwise substituting $(-1)$ for every $X_i$ leads to the negative trace $-1$. Therefore we may assume that $P = X_{\iota(1)}^{\tau(1)}\cdots X_{\iota(2k)}^{\tau(2k)}$ is trace nonnegative with $\ell$ variables and of degree $2k\geq 2$. 

We choose a family of edge-rooted graphs $\{(F_1,(a_1,b_1)),\cdots,(F_{\ell},(a_\ell,b_\ell))\}$ given by Proposition~\ref{prop:rigidfamily}.
Now let $H_1,\dots,H_{2k}$ be a family of disjoint graphs such that $H_i\cong F_{\iota(i)}^{\tau(i)}$ for every $i\in[2k]$. Consider the graph $G=G(H_1,\cdots,H_{2k})$ (defined in~\eqref{eq:Gsquare}) and the subgraph $C=G[v_1,\cdots,v_{2k}]$ which is isomorphic to the cycle of length $2k$. By Proposition~\ref{pro:cyclegluepositive}~\ref{itm:tracepos} we have that $G$ is positive. Therefore we can apply Lemma~\ref{lem:auxgraphsym} to $G$.

If Item \ref{itm:reflection} of the lemma applies, a cyclic permutation of $P$ can be written as $LL^T$. Thus $P$ is symmetric.

So assume that Item~\ref{itm:oneortwovar} of Lemma~\ref{lem:auxgraphsym} applies. Suppose that $P=(X_1^t)^{2k}$ for some $t\in\{1,T,\sym\}$. If $X_1$ is not necessarily symmetric, then $P=X_1^{2k}$ or $P=(X_1^T)^{2k}$. In this case choosing $M$ as in \eqref{eq:rotationm} but with $\theta=\pi/(2k)$ results in $\tr(P(M))=-2$ contradicting that $P$ is trace nonnegative. Therefore $P = (X_1^\sym)^{2k}$ which is a symmetric product. The second case is that $P=(X_1X_1^T)^{k}$ or $P=(X_1^TX_1)^{k}$ which are both symmetric. Lastly, if $P=(X_1^{t_1}X_2^{t_2})^{k}$ for some $t_1,t_2\in\{1,T,\sym\}$, then we choose the two symmetric matrices
\begin{equation}\label{eq:l1notpostr}
A\coloneqq\begin{bmatrix}
1 & 0\\
0 & -1
\end{bmatrix}\;\textrm{ and }\;B\coloneqq\begin{bmatrix}
\cos(\theta) & \sin(\theta)\\
\sin(\theta) & -\cos(\theta)
\end{bmatrix}
\end{equation}
where $\theta=\pi/k$. This shows that $\tr(P(A,B))=-2$ contradicting that $P$ is trace nonnegative.\\

Thus the leftover case is that Item~\ref{itm:periodic} of Lemma~\ref{lem:auxgraphsym} applies and in particular that the product consists only of symmetric variables. To address this scenario, note first that by Fact~\ref{Fact:AAt} our given product $P(X_1,\dots,X_\ell)=\prod_{i\in[k]}X_{\iota(i)}^\sym$ is trace nonnegative for any assignment of symmetric positive semidefinite matrices if and only if $ P'(Y_1, \dots, Y_\ell) = \prod_{i\in[k]}Y_{\iota(i)}Y_{\iota(i)}^T$ has a nonnegative trace for any assignment of real matrices. Therefore, since $P$ is trace nonnegative, it follows that $P'$ is also trace nonnegative.
Furthermore, note that what we did so far already proves the equivalence of Item~\ref{itm:realrealsym} and Item~\ref{itm:realrealtrace} of Theorem~\ref{thm:realreal}, since the only leftover case considers products which are restricted to symmetric variables. Since the variables $Y_i$ are not necessarily symmetric, we therefore know that $P'$ is trace nonnegative if and only if $P'$ is symmetric in terms of the variables $Y_1, \dots, Y_\ell$. Combining the observation that $P'$ is trace nonnegative with the proven implication of Theorem~\ref{thm:realreal} we get that $P' \eqshift L L^T$ where $L$ is a formal product in terms of variables $Y_i$ and $Y_i^T$. The remainder of the proof can be copied verbatim from the proof of Theorem~\ref{thm:realsym} which together with the observation that $P$ has even length shows that $P$ is symmetric
and thus completes the proof of Theorem~\ref{thm:realsym2}.
\end{proof}

As already discussed after its statement Theorem~\ref{thm:realreal} is implied by Theorems~\ref{thm:realsym} and~\ref{thm:realsym2}.
We have left to prove Theorems~\ref{thm:HilJoh},~\ref{thm:HilJoh2}, and~\ref{thm:HilJohno-sym} as applications of our main theorems. 

\begin{proof}[Proof of Theorem \ref{thm:HilJoh}]
Suppose first that $P$ is nearly symmetric. Then after substituting in any tuple of real symmetric positive semidefinite matrices $(A_1,\ldots,A_\ell)$ we have by Fact~\ref{Fact:ABBA} that $P(A_1,\ldots,A_\ell)$ has the same eigenvalues as some product $Q(A_1,\ldots,A_\ell)R(A_1,\ldots,A_\ell)$ where $Q$ and $R$ are both palindromes and therefore $Q(A_1,\ldots,A_\ell)$ and $R(A_1,\ldots,A_\ell)$ are both positive semidefinite. By the last claim in Fact \ref{Fact:AAt}, it follows that  $P$ is spectrally nonnegative and in particular trace nonnegative.

For the other direction suppose that $P$ is spectrally nonnegative or trace nonnegative over all positive semidefinite inputs. We substitute $X_i=Y_iY_i^T$ to get a product $\widetilde{P}$ and note that $\widetilde{P}$ is trace nonnegative over arbitrary real square matrices (by Fact~\ref{Fact:AAt}). Therefore by Theorem~\ref{thm:realreal} after a cyclic permutation $\widetilde{P}$ can be written as $LL^T$. This gives a reflection symmetry of the cyclic word $\widetilde{P}$. If the reflection axis cuts only between blocks $Y_iY_i^T$, then this reflection descends directly to the cyclic word $P$, showing that $P$ is symmetric up to cyclic permutation.

It remains to consider the case where the reflection axis cuts through the middle of a block $Y_iY_i^T$ (i.e., $Y_i$ is to the left and $Y_i^T$ is to the right). Suppose the other end of the reflection axis does not cut through $Y_i Y_i^T$ or any other block $Y_j Y_j^T$. Then a cyclic permutation of $P$ can be written as $X_i L L^T$ where $L$ is a product (word) of the variables. Thus, $P$ is nearly symmetric.

Now suppose the reflection axis cuts through $Y_iY_i^T$ on one end and through the middle of $Y_j Y_j^T$ on the other end (where $j$ may equal $i$). In this case, a cyclic permutation of $P$ can be written as $X_i L X_j L^T = X_i (L X_j L^T)$. This form is again nearly symmetric.
\end{proof}

\begin{proof}[Proof of Theorem~\ref{thm:HilJoh2}]
Suppose first that $P$ is nearly symmetric. Then after substituting in any tuple of real symmetric positive definite matrices $(A_1,\ldots,A_\ell)$ we have by Fact~\ref{Fact:ABBA} and Fact~\ref{Fact:AAt} that $P(A_1,\ldots,A_\ell)$ has only nonnegative real eigenvalues. Moreover, since the substitutions $A_i$ are in particular invertible, the matrix product $P(A_1,\ldots,A_\ell)$ is invertible. Hence none of its eigenvalues can be zero, so the nonnegative eigenvalues obtained above are in fact strictly positive. It follows that all eigenvalues of $P(A_1,\ldots,A_\ell)$ are real and positive and therefore $\operatorname{tr}(P)>0$.

Conversely, assume that $P$ always has strictly positive eigenvalues under substitutions by real symmetric positive definite matrices. Since the evaluation of the matrix word $P$ is continuous and the eigenvalues of a matrix depend continuously on its entries, passing to the topological closure shows that $P$ has only nonnegative eigenvalues under substitutions by real symmetric positive semidefinite matrices. Thus we may apply Theorem~\ref{thm:HilJoh} which implies that $P$ is nearly symmetric. The case where one assumes $\operatorname{tr}(P)>0$ follows by the analogous argument.
\end{proof}

\begin{proof}[Proof of Theorem~\ref{thm:HilJohno-sym}]
Item~\ref{itm:HJsym_sympower} implies Items~\ref{itm:HJsym_possp} and~\ref{itm:HJsym_postr} by Fact~\ref{Fact:ABBA}, Fact~\ref{Fact:AAt}, and the observation that the eigenvalues of $A^k$ are precisely the $k$-th powers of the eigenvalues of $A$.

Now, suppose that Item~\ref{itm:HJsym_possp} or~\ref{itm:HJsym_postr} holds. We make the formal substitution $X_i = B_i B_i^T C_i C_i^T$. Replacing each variable $X_i$ by this expression transforms $P$ into a new word $\tilde{P}$ in the variables $B_i$, $C_i$, and their transposes. By assumption and Fact~\ref{Fact:AAt}, $\tilde{P}$ evaluates to a matrix with nonnegative eigenvalues (respectively, has nonnegative trace) for all real square matrices $B_i$ and $C_i$.
We may therefore apply Theorem \ref{thm:realreal}, which implies that $\tilde P$ is symmetric. Consider the cyclic structure of $\tilde{P}$ and suppose its reflection axis cuts between some $B_1$ and $B_1^T$ (say, with $B_1^T$ positioned clockwise from $B_1$). This occurrence of $B_1 B_1^T$ must belong to an occurrence of $X_1$. By the structure of $X_1$, the adjacent block to its right (clockwise) must be $C_1 C_1^T$. Because $\tilde{P}$ is symmetric under this reflection, the term immediately to the left (counterclockwise) of the $B_1 B_1^T$ block must also be $C_1 C_1^T$. However, this new $C_1 C_1^T$ block on the left must belong to a different instance of $X_1$, and therefore must itself be preceded by a block $B_1 B_1^T$. By reflection, this forces yet another $B_1 B_1^T$ on the right, which in turn must be part of a new $X_1$ block, thereby forcing another $C_1 C_1^T$. Iterating this argument around the cycle, we conclude that the entire word is composed exclusively of repeating $X_1$ blocks. Thus, in this case $P = X_1^k$ for some integer $k$.

On the other hand, suppose the reflection axis does not cut strictly inside any $B_i B_i^T$ or $C_i C_i^T$ block. Then the reflection axis must lie between complete $X_i$ blocks, since it cannot lie between $B_i B_i^T$ and $C_i C_i^T$ as they are not reflections of one another. In this case, the reflection symmetry of $\tilde{P}$ directly implies that a cyclic permutation of the original word $P$ can be written in the form $L L^T$, where $L$ is a product of the original variables $X_i$ and their transposes. 

The case where nonnegativity is replaced by strict positivity follows from the same continuity argument as in the proof of Theorem~\ref{thm:HilJoh2}. 
\end{proof}

\section{Other Generalizations: Complex Matrices and Operators}\label{sec:generalizations}

The structural characterizations obtained in the previous sections are not specific to real matrices. They extend naturally to complex matrices once the formal transpose is replaced by the formal adjoint. More precisely, we consider matrix words in complex variables $\{X_i\}_{i \in [\ell]}$ together with a formal involution $*$ satisfying $(X_i^*)^*=X_i$ and $(X_iX_j)^*=X_j^*X_i^*$. If certain variables are required to be Hermitian, we denote them by $X_i^{\Herm}$ and impose the relation $(X_i^{\Herm})^* = X_i^{\Herm}$. A word $P=\prod_{i\in[k]}X_{\iota(i)}^{\tau(i)}$ is then defined by maps $\iota: [k] \to [\ell]$ and $\tau: [k] \to \{1, *, \Herm\}$. Furthermore, $P$ is called \emph{formally self-adjoint} if a cyclic permutation of $P$ can be written in the form $LL^*$, i.e. there exists an index $j\in[k]$ such that $\prod_{i=j+1}^{j+k}X_{\iota(i)}^{\tau(i)}=LL^*$, where indices are considered modulo $k$ and $L=\prod_{i=j+1}^{j+\tfrac{k}{2}}X_{\iota(i)}^{\tau(i)}$.

The following are the complex analogues of our results with proof which imply Corollary~\ref{cor:complex}, demonstrating that the algebraic rigidity of the real case perfectly mirrors the complex case.

\begin{coro}\label{thm:realsym-complex}
Let $P$ be a matrix word in complex variables and with some variables required to be Hermitian. Then $P$ is real-eigenvalued over complex matrices if and only if one of the following holds:
\begin{enumerate}
\item $P$ is formally self-adjoint, i.e. a cyclic permutation of $P$ is of the form $LL^*$; \label{itm:realsymsymcompl}
\item a cyclic permutation of $P$ is of the form $LL^*X_i^{\mathrm{Herm}}$. \label{itm:realsymsymplusvarcompl}
\end{enumerate}
\end{coro}

As in the real setting, the additional Hermitian factor occurring in the second case guarantees a real spectrum but does not force nonnegativity. Consequently, the classification simplifies when one requires all eigenvalues to be nonnegative.

\begin{coro}\label{cor:sympossemidefcom-complex}
Let $P$ be a matrix word in complex variables and with some variables required to be Hermitian. The following are equivalent.
\begin{enumerate}
    \item $P$ is formally self-adjoint.
    \item $P$ is spectrally nonnegative over complex matrices.
    \item $P$ is trace nonnegative over complex matrices. 
\end{enumerate}
\end{coro}

Second, we can  obtain a  characterization when restricting the domain entirely to (Hermitian) positive semidefinite matrices.  Consequently, we may consider the word $P$ as purely formed by variables $X_i$ without adjoints.

\begin{coro}\label{thm:HilJoh-complex}
Let $P$ be a matrix word. The following are equivalent:
\begin{enumerate}
\item $P$ is spectrally nonnegative over complex Hermitian positive semidefinite matrices.
\item $P$ is trace nonnegative over complex Hermitian positive semidefinite matrices.
\item $P$ is nearly symmetric; i.e., after a cyclic permutation, $P$ can be written as the product of two palindromes.
\end{enumerate}
\end{coro}

The same combinatorial characterization remains valid when nonnegativity is replaced by strict positivity, exactly as in Theorem~\ref{thm:HilJoh2}. The previous corollary concerns substitutions by Hermitian positive semidefinite matrices. As in the real setting, the Hermitian assumption on the substituted matrices can also be removed.

\begin{coro}\label{cor:posspecinput_complex}
    Let $P$ be a matrix word over complex variables $\{X_i\}_{i \in [\ell]}$. Then $P$ is spectrally nonnegative (or real-eigenvalued, or trace nonnegative) over arbitrary complex matrices with real nonnegative spectra if and only if $P$ is formally self-adjoint or $P$ is a power of a single variable (i.e., $P = X_i^k$).
\end{coro}

To prove these results in the complex setting, we first recall the following useful transformation between complex and real matrices. 
For \(A=U+iV\in \mathbb C^{n \times n}\), where \(U,V\in \mathbb R^{n \times n}\),
define its {\it realification} by
\[
\mathcal R_n(A):=
\begin{pmatrix}
U&-V\\
V&U
\end{pmatrix}
\in \mathbb R^{2n \times 2n}.
\]
The map \(\mathcal R_n\) is an injective unital
\(\mathbb R\)-algebra homomorphism and satisfies
\begin{equation}
\mathcal R_n(A^*)=\mathcal R_n(A)^T.
\label{7.1}
\end{equation}

In particular, \(\mathcal R_n\) sends Hermitian matrices to real
symmetric matrices and Hermitian positive semidefinite matrices to
real symmetric positive semidefinite matrices.
Given a complex matrix word \(P\), let \(P_{\mathbb R}\) denote the
corresponding real matrix word obtained by replacing each formal
adjoint \(*\) by the formal transpose \(T\), and each Hermitian
variable by a symmetric variable. Then
\begin{equation}
\mathcal R_n\bigl(P(A_1,\ldots,A_\ell)\bigr)
=
P_{\mathbb R}\bigl(
\mathcal R_n(A_1),\ldots,\mathcal R_n(A_\ell)
\bigr).
\label{7.2}
\end{equation}

Moreover, over \(\mathbb C\), the matrix \(\mathcal R_n(B)\) is
similar to \(B\oplus\overline B\). Consequently,
\begin{equation}
\operatorname{spec}\bigl(\mathcal R_n(B)\bigr)
=
\operatorname{spec}(B)
\cup
\overline{\operatorname{spec}(B)}.
\label{7.3}
\end{equation}
We have the following claim.
\begin{claim}\label{claim:complexrealequiv}
\(P\) is real-eigenvalued, respectively spectrally
nonnegative, over all admissible complex matrix substitutions of all
sizes if and only if \(P_{\mathbb R}\) has the corresponding property
over all admissible real matrix substitutions of all sizes. 
\end{claim}
\begin{proof}One
direction follows by restricting complex substitutions to real ones.
For the converse, apply the real hypothesis at size \(2n\) to the
realifications of an arbitrary complex \(n\times n\) tuple and use
(\ref{7.2}) and (\ref{7.3}). The doubling of dimension causes no difficulty because
the properties are required in every dimension.
\end{proof}

For completeness, the traces satisfy
\begin{equation}
\tr\bigl(\mathcal R_n(B)\bigr)
=
2\operatorname{Re}\tr(B).
\label{7.4}
\end{equation}
Thus realification alone transfers only the real part of the trace.
The complex trace-nonnegativity assertions below will therefore use
restriction to real matrices together with the resulting structural
characterization.

\begin{proof}[Proof of Corollaries \ref{thm:realsym-complex}, \ref{cor:sympossemidefcom-complex}]
Let \(P_{\mathbb R}\) be the real matrix word associated with \(P\)
under the realification correspondence above. Then by Fact \ref{claim:complexrealequiv}, Theorem \ref{thm:realsym} and 
Corollary \ref{cor:sympossemidef} give
Corollary \ref{thm:realsym-complex} and the equivalence of
Items~1 and~2 in Corollary \ref{cor:sympossemidefcom-complex}, after
replacing \(T\) by \(*\). 

For trace positivity, restriction to real substitutions and
Theorem \ref{thm:realsym2} show that
$P_{\mathbb R}=LL^T$ after circular permutation,
and hence \(P=LL^*\) after circular permutation. Conversely, this identity gives
$\tr P(A)
=\tr\!\left(L(A)L(A)^*\right)\geq0$
for every complex substitution \(A\). This proves the equivalence
of Items~1 and~3 in Corollary~\ref{cor:sympossemidefcom-complex}.
\end{proof}

\begin{proof}[Proof of Corollary \ref{thm:HilJoh-complex}]
   By realification, Item~1 is equivalent to its real counterpart, while
Item~2 implies its real counterpart by restriction. Hence Theorem~\ref{thm:HilJoh}
gives Item~3. Conversely, if \(P = L_1 L_2\) after circular permutation, where \(L_1\) and \(L_2\)
are palindromes, then on Hermitian positive semidefinite inputs both
\(L_1\) and \(L_2\) evaluate to positive semidefinite matrices. This gives Items~1
and~2. The strictly positive version follows similarly from
Theorem~\ref{thm:HilJoh2}. 
\end{proof}
\begin{proof}[Proof of Corollary \ref{cor:posspecinput_complex}]
    Realification preserves the class of matrices with spectrum contained
in \([0,\infty)\), so the real-eigenvalue and spectral-nonnegativity
assertions reduce to their real counterparts. Theorem~\ref{thm:HilJohno-sym}, and the
same proof using Theorem~\ref{thm:realreal} for real-eigenvaluedness, give the stated
results. For trace nonnegativity, necessity follows
by restriction to real matrices and Theorem~\ref{thm:HilJohno-sym}. Conversely, a word
cyclically equivalent to \(LL^*\) is spectrally nonnegative, while
\(X_i^k\) has eigenvalues \(\lambda^k\geq0\); in either case its trace
is nonnegative.
\end{proof}

\subsection{Bounded Operators on Hilbert Spaces}\label{subsec:operator}

Finally, this structural rigidity scales to infinite dimensions. Let $\mathcal{H}$ be an infinite-dimensional Hilbert space over $\mathbb{C}$, and let $\mathcal{B}(\mathcal{H})$ denote the set of bounded linear operators from $\mathcal{H}$ to itself. We consider variables $X_i$ taking values in $\mathcal{B}(\mathcal{H})$, where $X_i^*$ denotes the operator adjoint. 

A symbolic product of operators $P$ is called \emph{real-eigenvalued over $\mathcal{H}$} if, for any assignment of bounded operators, the resulting spectrum is entirely real. Because our definition of a formally self-adjoint word treats $P$ as a purely symbolic expression, it is independent of the underlying space and the definition carries over from the previous subsection. 
Any finite-dimensional matrix counterexample embeds into $\mathcal B(\mathcal H)$ by acting on a finite-dimensional subspace and extending by zero on the orthogonal complement. Thus universal positivity over $\mathcal B(\mathcal H)$ implies the finite-dimensional statement, while the converse follows from the explicit algebraic forms and basic operator theory.

\begin{coro}
    For every infinite-dimensional Hilbert space $\mathcal{H}$, let $X_i$ be variables for elements of $\mathcal{B}(\mathcal{H})$ and let $P$ be a symbolic operator word in $X_i$ and their adjoints $X_i^*$. Then $P$ is real-eigenvalued  over $\mathcal{H}$ if and only if $P$ is formally self-adjoint.
\end{coro}
Similarly, analogous results for spectral nonnegativity can be obtained. 
Furthermore, in the context of infinite-dimensional operator algebras, the concept of the normalized trace generalizes to that of a tracial state. Let $\mathcal{A}$ be a unital $C^*$-algebra or von Neumann algebra. A tracial state is a linear functional $\tau: \mathcal{A} \to \mathbb{C}$ that is positive ($\tau(A^*A) \ge 0$ for all $A \in \mathcal{A}$), normalized ($\tau(I) = 1$), and tracial ($\tau(AB) = \tau(BA)$ for all $A, B \in \mathcal{A}$). 
This setting serves as the foundation for the Connes Embedding Problem discussed in the introduction. Because a tracial state abstracts the cyclic invariance and positivity of the finite-dimensional matrix trace, our combinatorial characterization natively extends to this abstract setting.

\begin{coro}\label{cor:tracial_state}
    Let $P$ be a noncommutative word over complex variables and their adjoints. Then $P$ is formally self-adjoint if and only if $\tau(P(X_1, \dots, X_\ell)) \ge 0$ for all assignments of the variables in any $C^*$-algebra equipped with a tracial state $\tau$.
\end{coro}
\begin{proof}
    If $P$ is formally self-adjoint, then there exists a cyclic permutation of $P$ equating to $LL^*$. By the tracial property and positivity of the state, $\tau(P) = \tau(LL^*) = \tau(L^*L) \ge 0$, establishing the forward implication. The reverse implication holds trivially: since the property is assumed to hold for all $C^*$-algebras equipped with a tracial state, it must hold for the algebra of finite-dimensional complex matrices equipped with the standard normalized trace. By Corollary~\ref{cor:complex}, the word $P$ must be formally self-adjoint.
\end{proof}

\begin{coro}\label{cor:HilJoh-operator}
    Let $P$ be a noncommutative word over variables $\{X_1, \dots, X_\ell\}$. The following statements are equivalent:
    \begin{enumerate}
        \item For every infinite-dimensional Hilbert space $\mathcal{H}$, $P$ is spectrally nonnegative for all assignments of the variables to positive operators in $\mathcal{B}(\mathcal{H})$.
        \item For every unital $C^*$-algebra equipped with a tracial state $\tau$, $\tau(P) \ge 0$ for all assignments of the variables to positive elements in the algebra.
        \item $P$ is nearly symmetric; i.e., after a cyclic permutation, $P$ can be factored as the product of two palindromes.
    \end{enumerate}
\end{coro}
\begin{proof}
    We first show $(3)$ implies $(1)$ and $(2)$. Assume $P$ is nearly symmetric. Up to a cyclic permutation (which preserves both the non-zero spectrum and the tracial state), we can write $P = Q_1 Q_2$ for palindromic words $Q_1$ and $Q_2$. Positive operators and positive $C^*$-algebra elements are inherently self-adjoint ($X_i = X_i^*$). Evaluated on such variables, any palindrome naturally factors as $W W^*$ or $W X_k W^*$, meaning $Q_1$ and $Q_2$ strictly evaluate to positive elements. Because $Q_1 Q_2$ shares its non-zero spectrum with the positive element $Q_1^{1/2} Q_2 Q_1^{1/2}$, its spectrum is strictly nonnegative, yielding $(1)$. Similarly, the tracial property gives $\tau(Q_1 Q_2) = \tau(Q_1^{1/2} Q_2 Q_1^{1/2}) \ge 0$, establishing $(2)$.

    The reverse implications follow by restriction. If $(1)$ holds universally on $\mathcal{B}(\mathcal{H})$, we can embed any finite-dimensional complex Hermitian positive semidefinite matrix into $\mathcal{B}(\mathcal{H})$ by extending it by zero on the orthogonal complement, an embedding that trivially preserves the nonnegative spectrum. Likewise, if $(2)$ holds across all unital $C^*$-algebras, it inherently applies to the algebra of finite matrices $M_n(\mathbb{C})$ equipped with the standard normalized trace. In either case, the infinite-dimensional hypotheses reduce to the finite-dimensional conditions of Theorem~\ref{thm:HilJoh-complex}, forcing $P$ to be nearly symmetric.
\end{proof}

\section{Conclusion and Open Questions}

This paper gives complete combinatorial characterizations of universal
real-eigenvaluedness, spectral nonnegativity, and trace nonnegativity for
matrix words under each of the substitution classes considered.
A striking feature of the word setting is that, within each such
class, spectral and trace nonnegativity are governed by the same
combinatorial characterization, although the precise structure depends
on the permitted substitutions. This parallel rigidity is sharply
different from the situation for general noncommutative polynomials,
where spectral and trace positivity have substantially different
algebraic and algorithmic behavior. For matrix words, by contrast, the
relevant conditions are explicit and efficiently decidable. The
same rigidity persists over complex matrices and in the operator
settings considered. These results resolve questions first posed by
Lieb and Pierce and later conjectured by Hillar and
Johnson. Thus, matrix words admit an exact and effective positivity
theory despite the absence of a comparable general theory for
noncommutative polynomials.

A natural next step is to understand what remains true for sums of
words. Products of two sums of Hermitian squares form a basic class of
universally spectrally nonnegative polynomials. However, addition
destroys the cyclic structure used throughout this paper: cyclically
permuting individual monomials does not, in general, preserve the
spectrum of their sum. Any extension of our results on words therefore may involve a genuinely global polynomial certificate. Since an exact sum of squares characterization is known to fail for the trace positivity of general polynomials, our results for monomials completely resolve the algebraic picture for this property. Consequently, the search for a global algebraic certificate remains open only for the spectral nonnegativity of polynomials.

\begin{qu}
Can one characterize, or decide, the polynomials 
$p\in
\mathbb{R}\langle
X_1,\ldots,X_\ell,X_1^T,\ldots,X_\ell^T
\rangle$
such that
$\operatorname{spec}(p)
\subseteq [0,\infty)$
for every matrix size and every real matrix substitution?
In particular, is there an algebraic certificate for this property
analogous in role to Helton's sum-of-Hermitian-squares theorem?
\end{qu}

\section{Acknowledgment}
The second author would like to thank Dan Král' for hosting the second author's visit at Masaryk University in Brno where this work was initiated. The authors would like to thank Dan Kr\'al' and Jan Volec for helpful discussions. GPT 5.5 Pro and Gemini 3.1 Pro were used solely to find inconsistencies and to improve the presentation of the paper.

\bibliographystyle{plain}
\bibliography{DMP}

\end{document}